\documentclass[reqno,oneside,12pt]{amsart}

\usepackage[utf8]{inputenc}
\usepackage[T1]{fontenc}
\usepackage{times,mathptm,dsfont}
\usepackage{amssymb,epsfig,verbatim,xypic}
\usepackage{mathtools, stmaryrd}
\usepackage{amsthm}
\usepackage{amsmath}
\usepackage{latexsym,graphicx}
\usepackage{hyperref}

\theoremstyle{plain}

\newtheorem{thm}{Theorem}[section]
\newtheorem{cor}[thm]{Corollary}
\newtheorem{pro}[thm]{Proposition}
\newtheorem{lem}[thm]{Lemma}
\newtheorem{proposition-principale}[thm]{Proposition principale}
\newtheorem{thm-principal}{Main Theorem}

\newtheorem*{namedtheorem}{\theoremname}
\newcommand{\theoremname}{testing}
\newenvironment{named}[1]{\renewcommand{\theoremname}{#1}\begin{namedtheorem}}{\end{namedtheorem}}

\newtheorem*{namedtheoremr}{\theoremnamer}
\newcommand{\theoremnamer}{testing}

\theoremstyle{definition}

\newtheorem{rem}[thm]{Remark}

\newenvironment{defi-G}
{\noindent{\bf Definition.}\it}{\\}

\newenvironment{thm-M}
{\noindent{\bf Main Theorem.}\it }{}

\newenvironment{thm-A}
{\noindent{\bf Theorem A.}\it}{\\ }

\newenvironment{thm-B}
{\noindent{\bf Theorem B.}\it}{\\ }

\newenvironment{thm-BB}
{\noindent{\bf Theorem B'.}\it}

\newenvironment{thm-C}
{\noindent{\bf Theorem C.}\it}{\\ }

\def\C{\mathbf{C}}
\def\R{\mathbf{R}}
\def\Q{\mathbf{Q}}

\def\Z{\mathbf{Z}}
\def\N{\mathbf{N}}

\def\bfk{{\mathbf{k}}}
\newcommand{\grp}[1]{{\langle#1\rangle}}

\def\jj{{\sf{j}}}

\def\DA{\mathcal{A}}
\def\calr{\mathcal{R}}

\def\Coeff{\mathcal{C}}

\def\P{\mathbb{P}}

\def\disk{\mathbb{D}}

\def\U{{\mathcal{U}}}

\def\F{{\mathbf{F}}}

\def\A{\mathcal{A}}

\def\Diff{{\sf{Diff}}}
\def\Jet{{\sf{Jets}}}

\def\Hom{{\mathrm{Hom}}}
\def\Cont{{\sf{Cont}}}
\def\BX{{\mathbb{X}}}

\newcommand{\ad}{\mathrm{ad}}

\newcommand{\FD}{\widehat{\Diff}(\bfk,0)}
\newcommand{\norm}[1]{\lVert#1\rVert}

\def\Hyp{{\mathbb{H}}}

\def\PSL{{\sf{PSL}}\,}
\def\GL{{\sf{GL}}\,}

\newcommand{\Id}{{\rm Id}}
\newcommand{\ra}{{\to}}

\def\dist{{\sf{dist}}}

\addtocounter{section}{0}             
\numberwithin{equation}{section}       

\usepackage{color}

\newcommand{\eps}{\varepsilon}
\renewcommand{\phi}{\varphi}
\newcommand{\ol}[1]{\overline{#1}}
\newcommand{\m}{^{-1}}
\newcommand{\rad}{\mathrm{rad}}
\begin{document}

\setlength{\baselineskip}{0.56cm}        
%
%
\title[Surfaces groups in germs of diffeomorphisms]
{Surface groups in the group of germs of analytic diffeomorphisms in one variable.}
\date{}
\author{Serge Cantat, Dominique Cerveau, Vincent Guirardel, and Juan Souto}
\address{Univ Rennes, CNRS, IRMAR - UMR 6625, F-35000 Rennes, France}
\email{serge.cantat@univ-rennes1.fr}
\email{dominique.cerveau@univ-rennes1.fr}
\email{vincent.guirardel@univ-rennes1.fr}
\email{juan.souto@univ-rennes1.fr}
%
%

%
%

%
%

\begin{abstract} 
We construct embeddings of surface groups into the group of germs of analytic diffeomorphisms in one variable. 
\end{abstract}

\maketitle

\setcounter{tocdepth}{1}




\section{Introduction}

\subsection{The main result} Let $\C$ be the field of complex numbers and $\Diff(\C,0)$ the group of germs of analytic diffeomorphisms at the origin $0\in\C$. 
 Choosing a local coordinate $z$ near the origin, 
every element $f\in \Diff(\C,0)$ is determined by a unique power series 
\[
f(z)=a_1z+a_2z^2+a_3z^3+ \ldots + a_nz^n+\ldots
\]
with $f'(0)=a_1\neq 0$ and with a positive radius of convergence  
\begin{equation}\label{eq:radius}
\rad(f)=\left(\limsup_{n\to +\infty}\vert a_n\vert^{1/n}\right)^{-1}.
\end{equation}
We denote by $\Diff(\R,0)\subset \Diff(\C,0)$ the subgroup of real germs  
in this chart, i.e.\ with $a_i\in\R$ for all $i\in\N$ (this inclusion depends on the choice of the coordinate $z$).
The main goal of this note is the following result, that answers a 
question raised by  
 E.~Ghys (see \cite{Cerveau:Survey2013}, \S 3.3, or also \cite{Brudnyi:Survey2010}, Problem~4.15).

\begin{named}{Theorem A}
Let $\Gamma$ be the fundamental group of a closed orientable surface, or of a closed non-orientable surface of genus $\geq 4$. 
Then $\Gamma$ embeds in the group $\Diff(\R,0)$ and in particular in $\Diff(\C,0)$. 
\end{named}

We shall present three proofs of Theorem~A. 
For simplicity, in this introduction, we restrict to the case where $\Gamma$ is is the fundamental group of an orientable surface of genus 2, and
we consider the presentation
\begin{equation}\label{eq:presentation-G2-intro}
\Gamma_2=\grp{a_1,b_1,a_2,b_2\,|\ [a_1,b_1]=[a_2,b_2]}.
\end{equation}
Our proofs of theorem A are inspired by \cite{BGSS}, where it is proved that a compact topological group or a connected Lie group which contains a 
dense free group of rank $2$ contains a dense subgroup isomorphic to $\Gamma_2$.

Recently, and independently, A. Brudnyi proved a similar result for embeddings into the group of formal germs of diffeomorphisms (see~\cite{Brudnyi:preprint})

\subsection{Compact groups}\label{par:compact-groups}  
Let us describe the argument used in \cite{BGSS} to prove the following result. 

\begin{thm}[\cite{BGSS}]\label{thm:BGSS}
If a compact group $G$ contains a free group $F$ of rank $2$, then
there is an embedding $\rho\colon \Gamma_2\to G$ such that $F\subset \rho(\Gamma_2)$.
\end{thm}

\begin{proof} Denote by $\F_m$ the free group on $m$ generators. 
The first ingredient is a result by Baumslag \cite{Baumslag} saying that $\Gamma_2$ is {\bf{fully residually free}}; this means that there exists a sequence of morphisms $p_N:\Gamma_2\ra \F_2$ which is asymptotically injective: for every $g\in\Gamma_2\setminus \{1\}$, 
$p_N(g)\neq 1$ if $N$ is large enough. 

To be more explicit, we use the presentation~\eqref{eq:presentation-G2-intro} of $\Gamma_2$, and we note that the subgroup $\grp{a_1,b_1}$ of $\Gamma_2$
is a free group  $\F_2=\grp{a_1,b_1}$. Let $p:\Gamma_2\ra \grp{a_1,b_1}$ be the morphism fixing $a_1$ and $b_1$ and sending $a_2$ and $b_2$ to
 $a_1$ and $b_1$ respectively.
 Let $\tau:\Gamma_2\ra\Gamma_2$ be the Dehn twist around the curve $c=[a_1,b_1]$, i.e. the automorphism that fixes $a_1$ and $b_1$
 and send $a_2$ and $b_2$ to $c a_2 c\m$ and $c b_2 c\m$ respectively. 

\begin{pro}[see {\cite[Corollary~2.2]{BGSS}}] \label{pro_p1}
Given any $g\in \Gamma_2\setminus\{1\}$, there exists a positive integer $n_0$ such that 
$p\circ\tau^N(g)\neq 1$ for all $N\geq n_0$.
\end{pro}

Now, fix an embedding $\iota\colon \grp{a_1,b_1}\to G$ such that $\iota(\grp{a_1,b_1})=F$.
Composing $p\circ \tau^N$ with $\iota$, we get a sequence of points $p_N:=\iota\circ p\circ \tau^N$ in  $\Hom(\Gamma_2,G)$. 
Now, consider the element $h=\iota(p(c))$ of $G$, and let $T$ be the closure of the cyclic group $\grp{h}$ in the compact group $G$. 
For $t\in T$, define a morphism $\rho_t\colon \Gamma_2\to G$ by 
\begin{align}
\rho_t(a_1) &= \iota\circ p(a_1), & \rho_t(a_2) &= t\circ \iota\circ p(a_1) \circ t^{-1},\\
\rho_t(b_1) &= \iota\circ p(b_1), & \rho_t(b_2)&=t\circ \iota\circ p(b_1)\circ t^{-1};
\end{align}
 these representations are well defined and satisfy $\rho_t=\iota\circ p\circ\tau^N$ when $t=h^N$. Moreover, on the subgroup $\grp{a_1,b_1}$, 
  $\rho_{t}$ coincides with $\iota\circ p$, so  $F\subset \rho_t(\Gamma_2)$.
Thus, $(\rho_t)_{t\in T}$ is a compact subset $\calr(T)\subset \Hom(\Gamma_2,G)$
that contains the sequence of points $p_N$. For every $g$ in $\Gamma_2\setminus\{ 1\}$, the subset $\calr(T)_g=\{\rho_t\; \vert \; \rho_t(g)\neq 1\}$
is open, and Proposition~\ref{pro_p1}  shows that it is dense because $\{ h^n\; \vert \; n\geq n_0\}$ is dense in $T$ for every integer $n_0$. 
By the Baire theorem, the subset of injective representations $\rho_t$ is a dense $G_\delta$ in $\calr(T)$, and this proves Theorem~\ref{thm:BGSS}.
\end{proof}

The group $\Diff(\R,0)$ contains non-abelian free groups (this is well known, see Section~\ref{par:free-groups1}), 
and one may want to copy the above argument 
for $G=\Diff(\R,0)$ instead of a compact group. The Koenigs linearization theorem says that if $f\in\Diff(\R,0)$
satisfies  $f'(0)> 1$,
then $f$ is conjugate to the homothety $z\mapsto f'(0)z$; in particular, there is a flow of diffeomorphisms $(\phi^t)_{t\in \R}$ for which $\phi^1=f$.
In our argument, the compact group $T$ introduced to prove Theorem~\ref{thm:BGSS} will be replaced by such a flow, hence by a group isomorphic to $(\R, +)$. Also, in that proof,  $h=\iota(p(c))$ was a commutator, and the derivative of any commutator in $\Diff(\R,0)$ is equal to $1$ at the origin, so that Koenigs theorem can not be applied to a commutator. Thus,  we  need to change $p_N$ into  a different sequence of morphisms: the Dehn twist 
$\tau$ will be replaced by another automorphism of $\Gamma_2$, twisting along three non-separating curves.

This argument will be described in details in Sections~\ref{sec-formal} and~\ref{sec_orientable}; the reader who wants the simplest proof of Theorem~A in the case of orientable surfaces only needs to read these sections. Non orientable surfaces are dealt with in Section~\ref{sec_nonorientable}.

\subsection{Lie groups}
Now, let us look at representations in a linear algebraic subgroup $G$ of $\GL_m(\R)$. 
Assuming that there is a faithful representation $\iota\colon \F_2\to G$ with dense image, we shall construct a faithful representation $\Gamma_2\to G$. 

The representation variety $\Hom(\Gamma_2,G)$ is an algebraic subset of $G^4$. Let $\calr$ be the irreducible component containing the
trivial representation.  Let $p_N\colon \Gamma_2\to \F_2$ be an asymptotically injective sequence of morphisms, as given by Baumslag's proposition. When  the image of $\rho$ is dense,
one can prove that $\iota\circ p_N$ is in  $\calr$ for arbitrarily large values of $N$. For $g\in \Gamma_2\setminus\{ 1\}$, the subset $\calr_g\subset \calr$  of homomorphisms killing $g$ is  algebraic, and it is a proper subset because it does not contain $\iota\circ p_N$ for some large $N$.
Then, a Baire category argument in $\calr$ implies that a generic choice of $\rho\in \calr$ is faithful.

To apply this argument to $G=\Diff(\R,0)$, one needs a
good topology on $\Diff(\R,0)$, and a good ``irreducible variety'' $\calr\subset \Hom(\Gamma_2,G)$
containing $\iota\circ p_N$, in which a Baire category argument can be used.  
This approach may seem difficult because $\Hom(\Gamma_2,G)$ is a priori far from being an irreducible analytic variety but, again,  the  Koenigs linearization theorem
will provide the key ingredient. 

First, we shall adapt an idea introduced by Leslie in~\cite{Leslie} to define a useful group topology on $\Diff(\R,0)$ (see Section~\ref{sec-appendix}). 
With this topology, $\Diff(\R,0)$ is  an increasing union of Baire spaces, which will be enough
for our purpose.
Denote by $\Cont(\R,0)\subset \Diff(\R,0)$ the set of  elements $f\in\Diff(\R,0)$ with $|f'(0)|<1$; $\Cont$ stands for ``{\sl{contractions}}''.  
Consider the set $\calr$  of representations $\rho\colon \Gamma_2\to \Diff(\R,0)$ with $\rho(a_1)$ tangent to the identity, and $\rho(b_1)\in \Cont(\R,0)$.
Then, the key fact is that the map
\begin{align*}
\Psi:\calr&\ra \Cont(\R,0) \times \Diff(\R,0)\times\Diff(\R,0)\\
\rho&\mapsto (\rho(b_1),\rho(a_2),\rho(b_2))
\end{align*}
is a continuous bijection.
Indeed, the defining relation of $\Gamma$ is equivalent to $a_1b_1a_1\m=[a_2,b_2]b_1$.
Given $(g_1,f_2,g_2)\in \Cont(\R,0)\times \Diff(\R,0)\times\Diff(\R,0)$, 
the germs $g_1$ and $[f_2,g_2]g_1$ have the same derivative at the origin and, from the
Koenigs linearization theorem, there is 
a unique $f_1\in \Diff(\R,0)$ tangent to the identity solving the equation
$f_1g_1f_1\m=[f_2,g_2]g_1$: by construction there is a unique morphism $\rho\colon \Gamma_2\to \Diff(\R,0)$ 
that maps the $a_i$ to the $f_i$, and the $b_i$ to the $g_i$,  and this representation satisfies $\Psi(\rho)=(g_1,f_2,g_2)$.
With this bijection $\Psi$ and the topology of Leslie, we can identify $\calr$ with a union of Baire spaces, in which the Baire category argument applies.

\subsection{Other fields} 
 Let $\bfk$ be a finite field with $p$ elements. The group $\Diff^1(\bfk,0)$, also known as the Nottingham group,  is the group of power series tangent to the identity and with coefficients in the finite field $\bfk$. It is a compact group containing a free group (see \cite{Szegedy}). Thus, by~\cite{BGSS}, it contains a surface group.
 
\medskip

Now, let $p$ be a prime number, and let $\Q_p$ be the field of  $p$-adic numbers. 

Consider the subgroup $\Diff^1(\Z_p,0)\subset\Diff(\Q_p,0)$ of formal power series tangent to the identity and with coefficients in $\Z_p$.
First,  note that all elements $f$ of $\Diff^1(\Z_p,0)$ satisfy $\rad(f)\geq 1$, so that $\Diff^1(\Z_p,0)$ acts faithfully as a group of ($p$-adic analytic) homeomorphisms on $\{z\in \Z_p\; ;\; \vert z\vert <1\}$. So, in that respect, $\Diff^1(\Z_p,0)$ is much better than the group of germs of diffeomorphisms
$\Diff(\C,0)$. Moreover, with the topology given by the product topology on the coefficients $a_n\in \Z_p$ of the power series, the group $\Diff^1(\Z_p,0)$ becomes a compact group. And this compact group contains a free group. 
By the result of \cite{BGSS} described in Section~\ref{par:compact-groups}, it contains a copy of the surface group $\Gamma_2$. So, we get 
a surface group acting faithfully as a group of $p$-adic analytic homeomorphisms on $\{z\in \Z_p\; ;\; \vert z\vert <1\}$.
In Section~\ref{par:p-adic-proof} we give a third proof of Theorem~A that starts with the case of $p$-adic coefficients. 



\subsection{Organisation} The article is split in four parts. 
\begin{enumerate}
\item[{\bf{I.--}}] Sections~\ref{sec-formal} to~\ref{sec_nonorientable} give a first proof of Theorem~A; 
Section~\ref{sec_nonorientable}
is the only place where we deal with non-orientable surfaces. We refer to Theorem~B in 
Section~\ref{par:embeddings-other-fields-dense}  
for a stronger result, in which the field $\R$ 
is replaced by any non-discrete, complete valued field $\bfk$.

\item[{\bf{II.--}}] Section~\ref{sec-appendix} and~\ref{sec_Dominique} present our second proof, based on the construction of a group topology on $\Diff(\C,0)$. 

\item[{\bf{III.--}}] Then, our $p$-adic proof
is presented in Section~\ref{par:p-adic-proof}. 

\item[{\bf{IV.--}}] Section~\ref{sec_questions} draws some consequences and list a few open problems, while the appendix shows how to construct free groups in $\Diff(\C,0)$,
or $\Diff(\bfk,0)$ for any non-discrete and complete valued field. 
\end{enumerate}

\subsection*{Acknowledgement} Thanks to Yulij Ilyashenko, Frank Loray and Ludovic Marquis for several discussions on this subject,
and to the participants of the seminars of Dijon, Moscou, Paris, and Toulouse during which the results of this paper where presented. 

\tableofcontents


\medskip

\begin{center}
{\bf{ -- Part I. -- }}
\end{center}

\section{Germs of diffeomorphisms and the Koenigs Linearization Theorem}\label{sec-formal}

\subsection{Formal diffeomorphisms}\label{par:formal-diff-inversion}
Let $\bfk$ be a field (of arbitrary characteristic).
Denote by $\bfk\llbracket z\rrbracket$ the ring of formal power series in one variable with coefficients in $\bfk$. For every integer $n\geq 0$, let $A_n\colon\bfk\llbracket z\rrbracket\to\bfk$ denote the n-th coefficient function:
\begin{equation}\label{eq:coefficient-function}
A_n\colon f=\sum a_nz^n \; \mapsto \; A_n(f)=a_n.
\end{equation}
A {\bf formal diffeomorphism} is a formal power series $f\in \bfk\llbracket z\rrbracket$ such that $A_0(f)=0$ and $A_1(f)\neq 0$.
The composition $f\circ g$ determines a group law on the set
\begin{equation}
\widehat{\Diff}(\bfk,0)=\left\{f\in\bfk\llbracket z\rrbracket \; \vert \; A_0(f)=0\text{ and }A_1(f)\neq 0\right\}
\end{equation}
of all formal diffeomorphisms. 

For each $n\geq 1$, there is a polynomial $P_n\in \Z[A_1,B_1,\dots, A_n,B_n]$ 
such that if $f=\sum a_n z^n$ and $g=\sum b_n z^n$ 
then $f\circ g=\sum_{n\geq 1} P_n(a_1,b_1,\dots,a_n,b_n)z^n$.
Similarly, there  are polynomials $Q_n\in \Z[A_1,\dots, A_n][A_1\m]$ 
such that $f\m=\sum_{n\geq 1} Q_n(a_1,\dots,a_n)z^n$ if $f=\sum a_n z^n$; the polynomial function $Q_n$ 
is given by the following {\bf{inversion formula}}:
 \begin{equation*}\label{eq:inversion}
 \frac{1}{a_1^n}\sum_{k_1, k_2, \ldots} (-1)^{k_1+k_2+...} \cdot \frac{(n+1)\cdots(n-1+k_1+k_2+\ldots)}{k_1! \, k_2!\, \cdots} \cdot \left(\frac{a_2}{a_1}\right)^{k_1} \left(\frac{a_3}{a_1}\right)^{k_2}\cdots
 \end{equation*}
 where $a_i=A_i(f)$ and the sum is over all sequences of integers $k_i$ such that 
 \begin{equation*}\label{eq:coefficients-k}
 k_1+2k_2+3k_3+\cdots =n-1.
 \end{equation*}
We refer to \cite{Jennings} where this is proved for $f$ and $g$ tangent to the identity; the general case easily follows.

To encapsulate this kind of properties, we introduce the following definition. Let $m$ be a positive integer. 
By definition, a function $Q\colon  
\widehat{\Diff}(\bfk,0)^m\to \bfk$ is a {\bf{polynomial function}} with integer coefficients, if there is an integer $n$, and a polynomial $q\in \Z[A_{1,1},A_{2,1}, \ldots, A_{m-1,n}, A_{m,n} ][A_{1,1}^{-1}, \ldots, A_{m,1}^{-1}]$ such that 
\begin{align}
Q(f_1, \ldots, f_m) &= q(A_1(f_1), \ldots, A_n(f_m))
\end{align}
for all $m$-tuples $(f_1, \ldots, f_m)\in \widehat{\Diff}(\bfk,0)^m$; we denote by $\Z[\widehat\Diff(\bfk,0)^m]$ this ring of polynomial functions.

  Let $\F_m=\langle e_1,\dots,e_m\rangle$  be the free group of rank $m$. 
  To every word $w=e_{i_1}^{n_1}\ldots e_{i_k}^{n_k}$ in $\F_m$, we associate the {\bf{word map}}
  $ w:\widehat{\Diff}(\bfk,0)^m \to\widehat{\Diff}(\bfk,0)$,
  \begin{align}\label{eq-wordmap}
    (g_1,\dots,g_m)&\mapsto  w(g_1,\dots,g_m)\stackrel{\text{def}}{=}g_{i_1}^{n_1}\circ\ldots\circ g_{i_k}^{n_k}.
  \end{align}
Since composition and inversion are polynomial functions on $\widehat\Diff(\bfk,0)$, we get:
\begin{lem}\label{coefficientswordmap}
  Let $ w:\widehat{\Diff}(\bfk,0)^m \to\widehat{\Diff}(\bfk,0)$ be the word map given by some element of the free group $\F_m$. 
  For each $n\geq 1$, there is a polynomial function $Q_{w,n}\in \Z[\widehat\Diff(\bfk,0)^m]$ such that
$$ A_n(w(g_1, \ldots,g_m))=Q_{w,n}(g_1,\dots,g_m)$$
 for all $g_1,\dots,g_m\in\widehat\Diff(\bfk,0)$.
\end{lem}



\subsection{Diffeomorphisms and Koenigs linearization Theorem}\label{sec_Koenigs_formal}

Suppose now that $\bfk$ is endowed with an absolute value $\vert\cdot \vert \colon \bfk\to \R_+$. 
Then $\bfk$ becomes a metric space with the distance induced
by $\vert \cdot \vert$. We shall almost always assume that 
\begin{itemize}
\item $\bfk$ is not discrete, equivalently there is an element $x\in \bfk$ with $\vert x \vert \neq 0, 1$; 
\item $\bfk$ is complete.
\end{itemize}
Let $\bfk\{z\}$ be the subring of $\bfk\llbracket z\rrbracket$ consisting of power series $f(z)=\sum a_nz^n$ whose radius of convergence $\rad(f)$ is positive (see Equation \eqref{eq:radius}). 
When $\bfk$ is complete, the series $\sum a_nz^n$ converges uniformly in the closed disk $\disk_{r}=\{z\in\bfk\; \vert \; \vert z\vert\le r\}$ for every  $r< \rad(f)$. 
The group of germs of analytic diffeomorphisms is the intersection $\Diff(\bfk,0):=\widehat\Diff(\bfk,0)\cap\bfk\{z\}$; it is a subgroup of $\widehat\Diff(\bfk,0)$.

A germ $f\in\Diff(\bfk,0)$ is {\bf{hyperbolic}} if $|A_1(f)|\neq 1$. The following result is proved in 
\cite[Chapter 8]{Milnor:book} and \cite[Theorem 1, p. 423]{Herman-Yoccoz:1983} (see also \cite[Theorem 1]{Siegel} or  \cite{Koenigs}).

\begin{thm}[Koenigs linearization theorem]\label{thm:Koenigs} 
Let $(\bfk, \vert \cdot\vert)$ be a complete, non-discrete valued field. 
Let $f\in\Diff(\bfk,0)$ be a hyperbolic germ of diffeomorphism. 
There is a unique germ of diffeomorphism $h\in \Diff(\bfk,0)$ 
such that 
$
h(f(z))=A_1(f)\cdot h(z)
$ 
and $A_1(h)=1$.
\end{thm}

\section{Embedding orientable surface groups} \label{sec_orientable}

\subsection{Abstract setting}

Our strategy to construct embeddings of surface groups relies on the following simple remark.
Let $\Gamma$ be a countable group, and $G$ be any group.
Consider a non-empty topological space $\calr$, with a map $\Phi:s\in\calr\mapsto \Phi_s\in \Hom(\Gamma,G)$.
Given $g\in\Gamma$, set $\calr_g=\{s\in\calr \; \vert \;  \Phi_s(g)=1\}$.

\begin{lem}\label{lem_strategie}
  Assume that $\calr$ has the following 3 properties:
  \begin{enumerate}
  \item \emph{Baire:}  $\calr$ is a Baire space;
  \item \emph{Separation:} for every $g\neq 1$  in $\Gamma$, $\Phi_s(g)\neq 1$ for some $s\in \calr$;
  \item \emph{Irreducibility:} for every $g\in\Gamma$, either $\calr_g=\calr$ or $\calr_g$ is closed with empty interior.
  \end{enumerate}
Then the set of $s\in\calr$ such that $\Phi_s$ is an injective homomorphism is a dense $G_\delta$ in $\calr$; in particular, it is non-empty.
\end{lem}

\begin{proof}
  For any $g\in\Gamma\setminus\{1\}$, one has $\calr_g\neq \calr$ by (2), so $\calr_g$ is closed with empty interior by (3).
By the Baire property, $\calr\setminus (\cup_{g\in\Gamma\setminus\{1\}}\calr_g)$ is a dense $G_\delta$.
But $\calr\setminus (\cup_{g\in\Gamma\setminus\{1\}}\calr_g)$ is precisely the set of $s\in\calr$ such that $\Phi_s$ is injective.
\end{proof}

\subsection{Baumslag Lemma}\label{sec_Baumslag}

As explained in the introduction, it is proved in \cite{Baumslag} that the fundamental group of an orientable surface is  fully-residually free. 
We need a precise version of this result; to obtain it,
the main technical input is the Baumslag's Lemma (see {\cite[Lemma 2.4]{Olshanskii}}):
\begin{lem}[Baumslag's Lemma]\label{lem:BL}
 Let $n\geq 1$ be a positive integer. Let $g_0$, $\dots$, $g_n$ be elements of $\F_k$, and let $c_1$, $\dots$, $c_{n}$ be elements of $\F_k\setminus\{1\}$.
Assume that for all $1\leq i\leq n-1$, $g_i\m c_i g_i$ does not commute with $c_{i+1}$.
 Then for $N$ large enough, 
\[
g_0 c_1^N g_1 c_2^N\dots c_{n-1}^N g_{n-1} c_n^N g_n\neq 1.
\]
\end{lem}

 \begin{proof}[Sketch of proof (I)]  
 The group $\PSL_2(\R)$ acts on the hyperbolic plane $\Hyp$ by isometries, and contains a 
 subgroup $\Gamma$ such that (0) $\Gamma$ is isomorphic  to $\F_k$, (1) every element
 $g\neq \Id$ in  $\Gamma$ is a loxodromic isometry of $\Hyp$, and (2) two elements $g$ and $h$
 in $\Gamma\setminus\{\Id\}$ commute if and only if they have the same axis, if and only if they 
 share a common fixed point on $\partial \Hyp$. One can find such a group in any lattice of
 $\PSL_2(\R)$. To prove the lemma, we prove it in $\Gamma$. 
 
 Fix a base point $x\in \Hyp$, 
 denote by $\alpha_i$ and $\omega_i$ the repulsive and attracting fixed points of $c_i$ in 
 $\partial \Hyp$, and consider the word 
 \[
 g_0c_1^N g_1 c_2^N g_2.
 \]
 For $m$ large enough,  $c_2^m g_2$ maps $x$ to a point which is near $\omega_2$. 
 If $g_1(\omega_2)$ were equal to $\alpha_1$, then $c_1$ and $g_1c_2g_1^{-1}$ would share the common fixed point $\alpha_1$, and they would 
 commute. Thus, $g_1(\omega_2)\neq \alpha_1$ and then $g_0 c_1^{m'} g_1 c_2^m g_2$ maps $x$ 
 to a point which is near $g_0(\omega_1)$ if $m'$ is large enough. Thus,  $g_0 c_1^N g_1 c_2^N g_2(x)\neq x$ for large $N$. 
 The proof is similar if $n$ is larger than $2$.
 \end{proof}

 \begin{proof}[Sketch of proof (II)]  We rephrase this proof, using the action of $\F_k$ on its boundary, because this boundary will also be used in the proof of 
 Proposition~\ref{pro_twist3bords}. 
 
 Fix a basis $a_1, \ldots, a_k $ of $\F_k$, and denote by  $\partial \F_k$ the boundary of $\F_k$. The elements of $\partial \F_k$ are represented by infinite
 reduced words in the generators $a_i$ and their inverses. If
 $g$ is an element of $\F_k$ and $\alpha$ is an element of $\partial \F_k$ the concatenation $g\cdot \alpha$ is an element
 of $\partial \F_k$: this defines an action of $\F_k$ by homeomorphisms on the Cantor set $\partial \F_k$. If $g$ is a non-trivial, its
 action on $\partial \F_k$ has exactly two fixed points, given by the infinite words $\omega(g)=g \cdots g\cdots $ and $\alpha(g)=g^{-1} \cdots g^{-1}\cdots $
 (there are no simplifications if $g$ is given by a reduced and cyclically reduced word).
 Then we get: (1) every element
 $g\neq \Id$ in  $\F_k$ has a north-south dynamics on $\partial \F_k$, every orbit $g^n\cdot \beta$ converging to $\omega(g)$, except when $\beta=\alpha(g)$, and (2) two elements $g$ and $h$
 in $\F_k\setminus\{\Id\}$ commute if and only if they have the same fixed points, if and only if they 
 share a common fixed point on $\partial \F_k$. One can then repeat the previous proof with the action of $\F_k$ on its boundary.  \end{proof}

\begin{figure}[h]
\centering\epsfig{figure=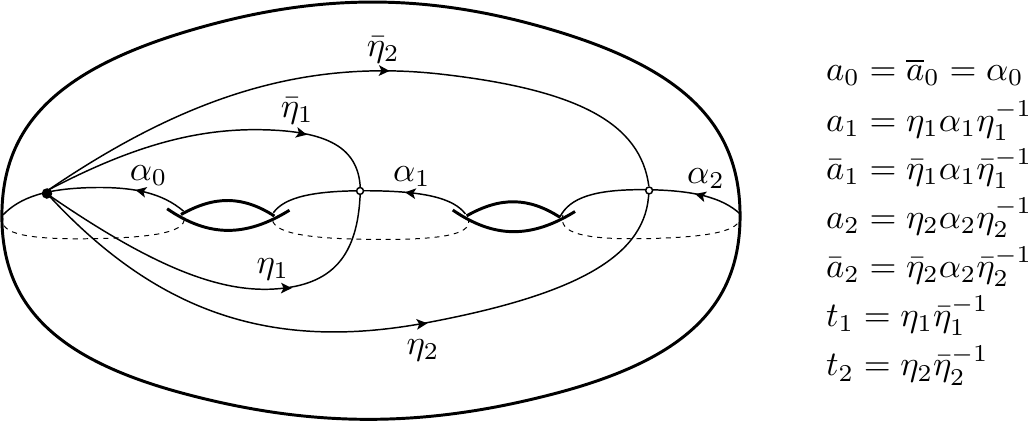}
\caption{{\small{The fundamental group $\Gamma_2$.-- The $\alpha_i$ are three loops, while the $\eta_j$ and ${\overline{\eta_j}}$ are 
four paths. The figure is symmetric with respect to the plane cutting the surface along the loops $\alpha_i$. 
}}}\label{fig:G2}
\end{figure}

Write the surface of genus $2$ as the union of two pairs of pants as in Figure \ref{fig:G2}, with respective fundamental groups 
\begin{equation}
\grp{a_0,a_1,a_2\; |\; a_0a_1a_2=1}
\; \text{ and } \; \grp{\ol a_0,\ol a_1,\ol a_2\; |\; \ol a_0 \ol a_1 \ol a_2=1}.
\end{equation}
This gives the presentation 
\begin{align}
\Gamma_2=\left\langle 
  \begin{matrix}
    a_0,a_1,a_2,\\ 
\ol a_0,\ol a_1,\ol a_2,\\
t_1,t_2
  \end{matrix}
\ \middle|\ 
  \begin{matrix}
    a_0a_1a_2=1,\\ \ol a_0\ol a_1\ol a_2=1,\\
\ol a_0=a_0,\ \ol a_1=t_1\m a_1 t_1,\ \ol a_2=t_2\m a_2 t_2 
  \end{matrix}
\right\rangle
\end{align}
which can be rewritten as
\begin{align}
\Gamma_2=\langle a_0,a_1,a_2,t_1,t_2\ |\ a_0a_1a_2=1,\ a_0 t_1\m a_1 t_1 t_2\m a_2 t_2=1\ \rangle.
\end{align}
Denote by  $p:\Gamma_2\ra \grp{a_0,a_1,a_2}\simeq \F_2$  the morphism defined by $p(a_i)=a_i$, 
$p(\ol a_i)=a_i$, and $p(t_1)=p(t_2)=1$.
Let $\tau:\Gamma_2\ra \Gamma_2$ be the (left) Dehn twist along the three curves $a_0$, $a_1$, and $a_2$, i.e. the automorphism
fixing $a_i$ and sending $t_i$ to $a_i t_i a_0\m$ for $i=1,2$.
Note the following facts:
\begin{itemize}
\item  $\tau$ sends $\ol a_i$ to $ a_0 \ol a_i a_0\m$; in particular, if $g$ is a word in the $\ol a_i$, then 
$\tau^N(g)=a_0^{N} g a_0^{-N}$;
\item $p\circ\tau^N$ fixes $a_i$ for every $i=0$, $1$, $2$, and 
\begin{equation}\label{eq:pcirctauN}
p\circ\tau^N(t_j)=a_j^Na_0^{-N}
\end{equation}
 for $j=1$, $2$.
\end{itemize}
\begin{pro}\label{pro_twist3bords}
For every $g\in \Gamma_2\setminus\{1\}$, there exists a positive  integer $n_0$ such that 
$p\circ\tau^N(g)\neq 1$ for all $N\geq n_0$.
\end{pro}

\begin{proof}
 To uniformize notations, we define $t_0=1$ so that for all $i\in\{0,1,2\}$ the relation $t_i \ol a_i t_i\m =a_i$
holds, and $\tau$ maps $t_i$ to $a_i t_i a_0\m$.
Let $A=\grp{a_0,a_1,a_2}$ and $\ol A=\grp{\ol a_0,\ol a_1,\ol a_2}$.
 Write $g$ as a shortest possible word of the following form:
\begin{equation}
g=g_0t_{i_1}g_1t_{i_2}\m g_2 t_{i_3} \dots g_{n-1}t_{i_n}\m g_n
\end{equation}
where $n$ is even, $i_k\in\{0,1,2\}$ for all $k\leq n$, $g_k\in A$ for $k$ even, $g_k\in \ol A$ for $k$ odd,
and the exponent 
of $t_{i_k}$ is $(-1)^{k+1}$ (we allow $g_k=1$).
One easily checks that $g$ can be written in this form because all generators can (for instance $t_1=1\cdot t_1 \cdot 1 \cdot t_0\m \cdot 1$).

If $k$ is such that $i_{k}=i_{k+1}$, then $g_k\notin \grp{a_{i_k}}$ if $k$ is even (resp $g_k\notin \grp{\ol a_{i_k}}$ if $k$ is odd)
as otherwise, one could shorten the word using the relation $t_{i_k} a_{i_k} t_{i_k}^{-1}=\ol a_{i_k}$.

\smallskip

{\bf{First claim.}} {\emph{If  $k\in\{2,\dots, n-2\}$ is even, $g_{k}\m a_{i_k} g_{k}$ does not commute to $a_{i_{k+1}}$}}.

\smallskip

If $i_k\neq i_{k+1}$, this is because $g_{k}\in A\simeq \F_2$ and no pair of $A$-conjugates of $a_{i_k}$ and $a_{i_{k+1}}$ commute.
If $i_k=i_{k+1}$, then $g_k\notin \grp{a_{i_k}}$ as we have just seen;
since $a_{i_k}$ is not a proper power in $A$, 
this shows that $g_k.(a_{i_k}^{+\infty})\neq a_{i_k}^{+\infty}$ in the boundary at infinity of the free group $A$,
so $g_k\m a_{i_k} g_k$ does not commute with $a_{i_k}$, and the claim follows.

\smallskip

Similarly, using the fact that $g_k\in\ol A$ for odd indices, we obtain: 

\smallskip

{\bf{Second claim.}} {\emph{
If $k\leq n-1$ is odd, $g_k\m \ol a_{i_k} g_k$ 
does not commute to $\ol a_{i_{k+1}}$}}.

\smallskip

We have $\tau^N(g_k)=g_k$ if $k$ is even, and $\tau^N(g_k)=a_0^Ng_k a_0^{-N}$ if $k$ is odd. After simplifications, one has
\begin{equation}
\tau^N(g)=g_0 a_{i_1}^{N}t_{i_1} g_1 t_{i_2}\m a_{i_2}^{-N} g_2 a_{i_3}^Nt_{i_3}  \dots g_{n-1}t_{i_n}\m a_{i_{n}}^{- N} g_n .
\end{equation}
For $k$ odd, denote by $g'_k\in \F_r$ the image of $g_k$ under $p$.
Applying $p$, we thus get
\begin{equation}
p\circ \tau^N(g)=g_0 a_{i_1}^{N} g'_1 a_{i_2}^{-N} g_2 a_{i_3}^N g'_3 \dots g'_{n-1} a_{i_{n}}^{- N} g_n,
\end{equation}
with $g'_i:=p(g_i)$.
Let us check that the hypotheses of the Baumslag Lemma~\ref{lem:BL} apply.
For $k$ even,  the first claim  shows that $g_{k}\m a_{i_k} g_{k}$ does not commute to $a_{i_{k+1}}$, as required.
For $k$ odd, we use that $p$ is injective on $\ol A$ and that $\ol A$ contains $g_k\m \ol a_{i_k} g_k$ and $\ol a_{i_{k+1}}$,
and we apply the second claim to deduce that  $g_k'^{-1}  a_{i_k} g'_k$ does not commute to $a_{i_{k+1}}$.
Applying Baumslag's Lemma, we conclude that $p\circ\tau^N(g)\neq 1$ for $N$ large enough.  
\end{proof}

\subsection{Embeddings of free groups} \label{par:free-groups1}
The group $\Diff(\R,0)$ contains non-abelian free groups. 
This has been proved by arithmetic means \cite{White:1988,Glass:1992}, 
by looking at the monodromy of generic polynomial planar vector fields \cite{IP}, 
and by a dynamical argument \cite{MRR}. We shall need the following precise version of that result.

\begin{thm}\label{thm_F2} 
 Let $(\bfk, \vert \cdot\vert)$ be a complete non-discrete valued field.
 For every pair $(\lambda_1,\lambda_2)$ in $\bfk^*$, 
  there exists a pair $f_1$, $f_2\in\Diff(\bfk,0)$ that generates a free group
and satisfies $f_1'(0)=\lambda_1$ and $f_2'(0)=\lambda_2$.
\end{thm}

This result is proved in \cite[Proposition 4.3]{BCN} for generic pairs of derivatives $(\lambda_1,\lambda_2)$. We
provide a proof of Theorem~\ref{thm_F2}  in the Appendix, extending the argument of \cite{MRR}. We refer to 
Section~\ref{par:free-group-Z} below for other approaches.

\subsection{Embedding orientable surface groups}
We can now prove  Theorem~A for orientable surfaces:

\begin{thm}\label{thm_g2}
Let   $\Gamma_g$ be the fundamental group of a closed, orientable surface of genus $g$. 
Then, there exists an injective morphism $\Gamma_g\ra \Diff(\R,0)$.
\end{thm}

The group $\Gamma_0$ is trivial. The group $\Gamma_1$ is isomorphic to $\Z^2$, hence it embeds in the 
group of homotheties $z\mapsto \lambda z$, $\lambda\in \R_+^*$. If $g\geq 2$, then $\Gamma_g$ embeds in $\Gamma_2$. To see this, fix a surjective morphism $\Gamma_2\to \Z$, and take the preimage $\Lambda\subset \Gamma_2$ of the subgroup $(g-1)\Z\subset \Z$. Then, $\Lambda$ is a  normal subgroup of index $g-1$ in $\Gamma_2$, and  it is the fundamental group of a closed surface $\Sigma$, given by a Galois cover of degree $g-1$ of the surface of genus $2$. Since the Euler characteristic is multiplicative, the genus of $\Sigma$ satisfies $-2(g-1)=2-2g(\Sigma)$. Thus, $g(\Sigma)=g$ and $\Lambda$ is isomorphic to $\Gamma_g$. Thus,  we now restrict to the case $g=2$.

\smallskip

By Theorem~\ref{thm_F2}, we can fix an injective morphism
\begin{equation}\label{eq_rho0}
\rho_0:\F_2=\grp{a_0,a_1,a_2\; |\; a_0a_1a_2=1}\ra \Diff(\R,0)
\end{equation}
such that the images $f_1=\rho_0(a_1)$, $f_2=\rho_0(a_2)$, and   $f_0=\rho_0(a_0)=f_2\m f_1\m$ satisfy 
\begin{equation}
 f_1'(0)=\lambda_1>1, \; \;  f_2'(0)=\lambda_2>1,\; \;  f_0'(0)=\lambda_0<1
\end{equation}
for some real numbers $\lambda_1$ and $\lambda_2>1$ and $\lambda_0=(\lambda_1\lambda_2)^{-1}$.
In particular, $f_0$, $f_1$, and $f_2$ are hyperbolic. 
For $\lambda\in\R^*$, denote by $m_\lambda(z)=\lambda z$ the corresponding homothety.
For $i\in\{0,1,2\}$,
the Koenigs linearization theorem shows that $f_i$ is conjugate to the homothety $m_{\lambda_i}$: there is a germ 
of diffeomorphism $h_i\in \Diff(\R,0)$ such that 
$f_i=h_i\circ m_{\lambda_i}\circ h_i\m$.
Thus $f_i$ extends to the multiplicative flow $\phi_i:\R_+^*\ra \Diff(\R,0)$ defined by
$\phi_i^s=h_i\circ m_{s} \circ h_i\m$ for $s\in\R_+^*$; by contruction, $\phi_i^{\lambda_i}=f_i$ and 
$\phi_i^s$ commutes with $f_i$ for all $s>0$.
We note that $s\mapsto \phi_i^s$ is polynomial in the sense that for all $k\in\N$,
$s\mapsto A_k(\phi_i^s)$ is a polynomial function with real coefficients in the variables $s$ and $s\m$.

Set $\calr=(\R_+^*)^3$. As in Section~\ref{sec_Baumslag}, consider the presentation 
\begin{align}
\Gamma_2=\langle a_0,a_1,a_2,t_1,t_2\ |\ a_0a_1a_2=1,\ a_0 t_1\m a_1 t_1 t_2\m a_2 t_2=1\ \rangle.
\end{align}
Given $s=(s_0,s_1,s_2)\in(\R_+^*)^3$, we define a  morphism $\Phi_s:\Gamma_2 \to \Diff(\R,0)$ by
\begin{align}
\Phi_s(a_i) & = f_i\text{\quad for $i\in\{0,1,2\}$} \\
\Phi_s(t_i)& = \phi_i^{s_i}\phi_0^{s_0}\text{\quad for $i\in\{1,2\}$} 
\end{align}
This provides a well defined homomorphism because $\phi_i$ commutes with $f_i$. As we shall see below, this morphism $\Phi_s$ 
is constructed to coincide with $\rho_0\circ p\circ \tau^N$ for $s=(\lambda_0^N,\lambda_1^N,\lambda_2^N)$ (see~Equation~\eqref{eq:pcirctauN}).

\begin{rem}\label{rem_g2} For every $s\in \calr$, the image of $\Phi_s$ contains $f_1$ and $f_2$, hence the free group $\rho_0(\F_2)$. This
will be used in Section~\ref{par:embeddings-other-fields-dense}. \end{rem}

To conclude, we  check that the three assumptions of Lemma~\ref{lem_strategie} hold for this family of 
morphisms $(\Phi_s)_{s\in \calr}$.

Clearly, $\calr$ is a Baire space.

To check the irreducibility property, consider $g\in\Gamma_2$ and assume
that $\calr_g\neq \calr$: this means that  there exists a parameter $s\in\calr$ and an index 
 $k\geq 1$  such that $A_k(\Phi_s(g))\neq A_k(\Id)$.
The map $s=(s_0,s_1,s_2)\mapsto A_k(\Phi_s(g))-A_k(\Id)$ is a polynomial function in the variables $s_0^{\pm 1}$, $s_1^{\pm1}$, and $s_2^{\pm 1}$
that does not  vanish identically on $\calr$, so its zero set is a closed subset with empty interior.

We now check that $\calr$ has the separation property. As in Section~\ref{sec_Baumslag}, denote by $p:\Gamma_2\ra \F_2=\grp{a_0,a_1,a_2 \; |\; a_0a_1a_2=1}$  the morphism obtained by killing $t_1$ and $t_2$. For the parameter $s=(1,1,1)$, $\Phi_{s}$ is equal to $\rho_0\circ p$. 
More generally,  setting $s_N=(\lambda_0^N,\lambda_1^N,\lambda_2^N)$ for $N\in\N$, the morphism $\Phi_{s_N}:\Gamma_2 \to \Diff(\R,0)$
satisfies 
\begin{align}
\Phi_{s_N}(a_i)& = f_i\text{\quad for $i\in\{0,1,2\}$} \\
\Phi_{s_N}(t_i)&= \phi_i^{u_i^N}\phi_0^{u_0^N}=f_i^Nf_0^N\text{\quad for $i\in\{1,2\}$}. 
\end{align}
This means that $\Phi_{s_N}=\rho_0\circ p\circ \tau^N$ 
where, as in Section \ref{sec_Baumslag}, $\tau:\Gamma_2\ra \Gamma_2$ is the Dehn twist along the three curves $a_i$.
By Proposition \ref{pro_twist3bords}, for all $g\in\Gamma_2\setminus\{1\}$ there exists $N\in\N$ such that
$p\circ \tau^N(g)\neq 1$.
Since $\rho_0$ is injective, this implies that $\Phi_{s_N}(g)\neq 1$ which shows that $\calr$ has the separation property.

\section{Non-orientable surface groups}\label{sec_nonorientable}

\begin{thm}\label{thm_Ng}
Let $N_g$ be the fundamental group of a closed non-orientable surface of genus $g\geq 4$.
  There exists an injective morphism $N_g\ra \Diff(\R,0)$.
\end{thm}

\begin{rem}
  The fundamental group $N_3$ of the
  non-orientable surface of genus
  $3$ is not fully residually free,
  and our methods do not apply to
  this group. (See~\cite{Lyndon:1959}, Proposition 9.)
\end{rem}
 
\subsection{Even genus}

We first treat the case of an even genus $g\geq 4$. 
In this case, the group $N_{g}$ embeds in $N_4$. Indeed, the non-orientable surface of genus $4$ is 
the connected sum of a torus $\R^2/\Z^2$ with two projective planes $\P^2(\R)$. Taking a cyclic cover 
of the torus of degree $k$, we get a surface homeomorphic to the connected sum of $\R^2/\Z^2$ with $2k$ copies of
$\P^2(\R)$, hence a non-orientable surface of genus $2(k+1)$.  Thus, it suffices to prove that $N_4$ embeds in $\Diff(\R,0)$.

The non-orientable surface of genus $4$ is homeomorphic to the connected sum of $4$ copies of $\P^2(\R)$, and this gives the presentation (see Figure~\ref{fig:N4})
\begin{equation}
N_4=\langle a_1,a_2, b_1, b_2 \; |\;  a_1^2 a_2^2 b_2^2 b_1^2 =1 \rangle.
\end{equation}

Let $p:N_4\ra \grp{a_1,a_2}$ be the morphism fixing $a_1,a_2$ and sending $b_1$ and $b_2$ to
$a_1\m$ and $a_2\m$ respectively.
Let $\tau:N_4\ra N_4$ be the Dehn twist around the curve $\gamma=(a_1^2 a_2^2)^{-1}$, i.e.\ the automorphism that fixes $a_1$ and $a_2$
and sends $ b_1$ and $b_2$ to $\gamma b_1 \gamma\m$ and  $\gamma b_2 \gamma\m$ respectively.

\begin{figure}[h] 
\centering\epsfig{figure=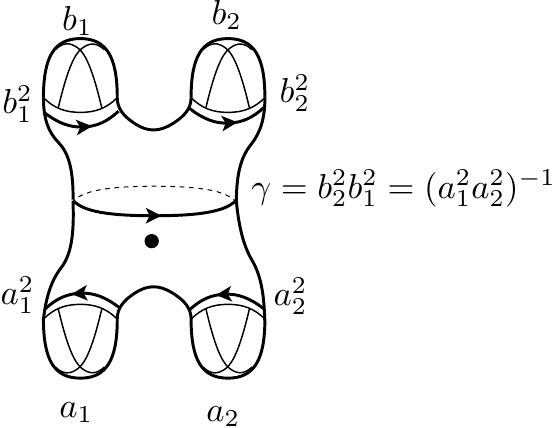}
\caption{{\small{The fundamental group $N_4$.-- The base point is represented by $\bullet$, the $4$ generators are $a_1$, $a_2$, $b_1$, $b_2$, and the curve $\gamma$ is used to construct the Dehn twist $\tau$. }}}\label{fig:N4}
\end{figure}

\begin{lem}\label{lem_p2}
Given any $g\in N_4\setminus\{1\}$, there exists $n_0\in\N$ such that for all $N\geq n_0$,
$p\circ\tau^N(g)\neq 1$.
\end{lem}

\begin{proof}[For the proof] The proof of this statement is completely analogous to the proof of \cite[Corollary~2.2]{BGSS}, using Baumslag Lemma,
we leave it as an exercise to the reader. See also \cite[Proposition 4.13]{CG}. \end{proof}

Using Theorem \ref{thm_F2}, we fix two germs of diffeomorphisms $f_1$ and $f_2\in\Diff(\R,0)$ generating a free group and satisfying
$f'_1(0)>1$ and $f'_2(0)>1$.
We denote by 
\begin{equation}
\rho_0:\F_2=\grp{a_1,a_2}\ra \Diff(\R,0)
\end{equation} the injective morphism sending $a_i$ to $f_i$ for $i\in\{1,2\}$.
In particular, 
\begin{equation}
\rho_0(\gamma)=(f_1^2\circ f_2^2)^{-1}
\end{equation}
 is a hyperbolic germ: its derivative $\lambda=((f_1^2\circ f_2^2)'(0))^{-1}$ is $<1$.
The Koenigs linearization theorem gives an element $h\in\Diff(\R,0)$ such that $\rho_0(\gamma)=h\circ m_\lambda \circ h\m$.
Consider the multiplicative flow $\phi:\R_+^*\ra \Diff(\R,0)$ defined by $\phi^s=g\circ m_s \circ g\m$.
As above, $\phi^\lambda=\rho_0(\gamma)$, $\phi^s$ commutes with $\rho_0(\gamma)$ for all $s>0$, and 
$s\mapsto \phi^s$ is a polynomial map: for all $k\in\N$,
$s\mapsto A_k(\phi^s)$ is a polynomial in the variables $s$ and~$s\m$.

Set $\calr=\R_+^*$. Given $s\in\R_+^*$, consider the morphism $\rho_s:N_4\to \Diff(\R,0)$ defined by
\begin{align*}
a_1&\mapsto f_1 & a_2&\mapsto f_2\\
b_1&\mapsto \phi^{s}f_1\m \phi^{-s} & b_2&\mapsto \phi^{s}f_2\m \phi^{-s}.
\end{align*}
This gives a well defined homomorphism because $\phi^s$ commutes with $f_1^2f_2^2$.

We now check the three assumptions of Lemma \ref{lem_strategie}.
Clearly, $\calr$ is a Baire space.
The irreducibility is a consequence of the fact that for any $g\in N_4$, and any $k\in\N$ the map $s\mapsto A_k(\phi_s(g))$
is a polynomial function in the variables $s^{\pm1}$.
The separation property follows from Lemma \ref{lem_p2}
together with the fact that $\rho_{\lambda^N}=\rho_0\circ p\circ \tau^N$ and that $\rho_0$ is injective.

\subsection{Odd genus}

We now treat the case of a non-orientable surface of odd genus $g=2k+1$, $k\geq 2$.
One can write $N_{2k+1}$ as (see Figure~\ref{fig:Nodd} below) 
\begin{equation}
N_{2k+1}=\grp{a_1,\dots,a_k,c,b_1,\dots,b_k\ |\ a_1^2\dots a_k^2c^2b_k^2\dots b_1^2=1}.
\end{equation}
This group splits as a double amalgam of free groups
\begin{equation}\label{eq:dble-amalgam}
N_{2k+1}=\grp{a_1,\dots a_k}\underset{a_1^2\dots a_k^2=\gamma\m}{*}\grp{\gamma,c}\underset{c^{-2}\gamma=b_k^2\dots b_1^2}{*} \grp{b_1,\dots,b_k}.
\end{equation}
We shall use the following notation to refer to this amalgam structure:
\begin{itemize}
\item $A_1=\grp{a_1,\dots a_k}$ and $e_{1,2}=(a_1^2\dots a_k^2)^{-1}$;
\item $A_2=\grp{\gamma,c}$ and $e_{2,1}=\gamma$ and $e_{2,3}=c^{-2}\gamma=\delta$; 
\item $A_3=\grp{b_1,\dots,b_k}$ and $e_{3,2}=b_k^2\dots b_1^2$.
\end{itemize}
So, each of the $A_i$ is a free group and the amalgamation is given by $e_{1,2}=e_{2,1}$
and $e_{2,3}=e_{3,2}$.

Define a morphism $p\colon  N_{2k+1}\ra \grp{a_1,\dots,a_k}\simeq \F_k$  by 
\begin{align*}
a_i&\mapsto a_i \text{\quad for $i\leq k$} & c&\mapsto a_k^{-2}\\
b_i&\mapsto a_i\m \text{\quad for $i\leq k-1$} &b_k&\mapsto a_k 
\end{align*}
(the structure of almagam shows that $p$ is well defined).

\begin{figure}[h]
\centering\epsfig{figure=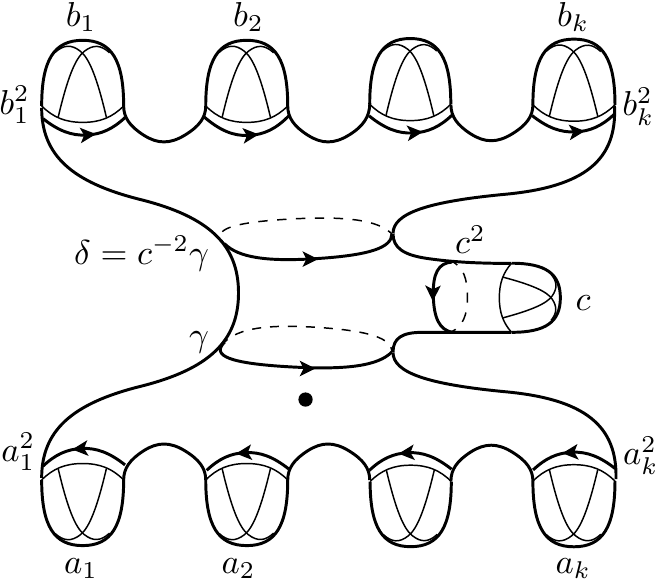}
\caption{{\small{The fundamental group $N_{2k+1}$.}}}\label{fig:Nodd}
\end{figure}

\begin{lem}
The morphism $p\colon N_{2k+1}\ra   \F_k$ is injective in restriction to each of the three subgroups of the amalgam~\eqref{eq:dble-amalgam}.
\end{lem}
\begin{proof} By construction, it is injective in restriction to $\grp{a_1,\dots , a_k}$ and in restriction to $\grp{b_1,\dots , b_k}$.
Then, note that $p(\grp{\gamma,c })=\grp{a_1^2\dots a_k^2,a_k^{-2}}$ is isomorphic to $\F_2$ because it is a non-abelian subgroup of a free group.
Since $\F_2$ is Hopfian, $p$ is necessarily injective in retriction to $\grp{\gamma,c}$.
\end{proof}

Consider $\delta=b_k^2\cdots b_1^2=c^{-2}\gamma$ and note that $p(\delta)=a_k^{2}a_{k-1}^{-2}\cdots a_1^{-2}$.
Let $\tau$ be the Dehn twist corresponding to the decomposition above, i.e.\ the automorphism fixing $a_i$, sending $c$ to $\gamma c \gamma\m$
and sending  $b_i$ to $(\gamma\delta) b_i (\gamma\delta)\m$. Since $\tau$ is the composition of the twists
given by $\gamma$ and $\delta$ and these two twists commute we get 
\[
\tau^N(b)=(\gamma^N\delta^N) b (\gamma^N\delta^N)^{-1}, \quad \forall b\in A_3.
\]
In this situation, one can prove the following lemma in a similar way to Proposition \ref{pro_twist3bords}.

\begin{lem}\label{lem_p3}
Given any  $g\in N_{2g+1}\setminus\{1\}$, there exists $n_0\in\N$ such that for all $N\geq n_0$,
$p\circ\tau^N(g)\neq 1$.
\end{lem}

 \begin{proof}  
  Write $g$ as a word in the graph of groups, i.e.\
 $g=s_0\dots s_n$ with $s_k\in A_{r_k}$ (we allow $s_k=1$) for some $r_k\in\{1,2,3\}$,
 with $r_{k+1}=r_k\pm 1$, and $r_0=r_n=1$. We take this word of minimal possible length among words satisfying these contraints.
 If $k$ is such that $r_{k-1}=r_{k+1}$, then $s_k\notin \grp{e_{r_k, r_{k+1}}}$ since otherwise, one could shorten the word using the
 structure of amalgam  (in particular $s_k\neq 1$ in this case).
 Now one easily checks that
 \begin{equation}
 \tau^N(g)=s_0 d_{1}^{\eps_1 N}s_1 d_{2}^{\eps_2 N} s_2 \cdots d_{n}^{\eps_{n} N} s_n 
 \end{equation}
 where $d_{k}=e_{r_{k-1},r_k}\in \{\gamma,\delta\}$, and $\eps_k= r_{k}-r_{k-1}\in\{\pm 1\}$.
 
 We claim that
 $s_k\m d_{k} s_k$ does not commute with $d_{k+1}$.
 If $d_{k}\neq d_{k+1}$, this follows from the fact that $\gamma$ commutes with no conjugate of $\delta$ in $A_2=\grp{c,\delta}$.
 If $d_{k}=d_{k+1}$, then $r_{k-1}=r_{k+1}$, so $s_k\notin \grp{e_{r_k, r_{k+1}}}=\grp{d_k}$.
 If $[s_k\m d_ks_k,d_k]=1$, then $s_k$ preserves the axis of $d_k$ in the Cayley graph of the free group $A_{r_k}$, so $s_k$ is a power of $d_k$,
 because $d_k\in \{ \gamma, \delta\}$ is not a proper power; this contradicts that $s_k\notin \grp{d_k}$.

 Denote by $\ol s_k$, $\ol d_k\in \F_r$ the images of  $s_k$, $d_k$ under $p$.
 Since $p$ is injective on each $A_{r_k}$,
  $\ol s_k\m \ol d_{k} \ol s_k$ does not commute with $\ol s_{k+1}$,
 so the hypotheses of Baumslag Lemma apply to the word
 \begin{equation}
 p\circ\tau^N(g)=\ol d_{0}^{\eps_0 N}\ol s_1 \ol d_{1}^{\eps_1 N} \ol s_2 \dots \ol d_{n-1}^{\eps_{n-1} N} \ol s_n \ol d_{n}^{\eps_n N}
 \end{equation}
 so $p\circ\tau^N(g)\neq 1$ for $N$ large enough.
 \end{proof}

Now consider $k$ elements $f_1$, $\dots$, $f_k$ of $\Diff(\R,0)$ generating a free group of rank $k$ with $f'_i(0)>1$ for all $i\in\{1,\dots,k\}$,
and $f'_k(0)<f'_1(0)$. Such a set can be obtained from two generators $g_1$ and $g_2$ of a free group of rank $2$ with $g_i'(0)>1$, as in Theorem \ref{thm_F2}, by taking $f_i=g_1^i\circ g_2^2\circ g_1^{-i}$
for $i<k$ and $f_k=g_1^k\circ g_2\circ g_1^{-k}$.
Let $\rho_0:\F_k=\grp{a_1,\dots, a_k}\ra \Diff(\R,0)$ be the injective morphism sending $a_i$ to $f_i$ for $i\leq k$.
In particular, $\rho_0(\gamma)=(f_1^2\circ \dots \circ f_k^2)\m$ and $\rho_0(p(\delta))=f_k^2 \circ f_{k-1}^{-2}\circ\dots\circ f_1^{-2}$ are hyperbolic.
Using Koenigs linearization theorem as above, there exists two multiplicative flows
$\phi$ and $\psi:\R_+^*\ra \Diff(\R,0)$ 
 and a pair of positive real numbers $\lambda$ and $\mu$ such that (1) $\phi^\lambda=\rho_0(\gamma)$ and $\psi^\mu=\rho_0(p(\delta))$,
and (2) $s\mapsto \phi^s$ and $s\mapsto \psi^s$ are polynomial mappings.

Set $\calr=(\R_+^*)^2$ and, for every $(s,s')\in\calr$, define 
a morphism $\rho_{s,s'}:\N_{2k+1}\ra \Diff(\R,0)$ by
\begin{align*}
a_i&\mapsto f_i \text{\quad for $i\leq k$} & c&\mapsto \phi^s f_k^{-2} \phi^{-s} \\
b_i&\mapsto \phi^s\psi^{s'}f_i\m (\phi^s\psi^{s'})\m \text{\quad for $i\leq k-1$} & b_k&\mapsto \phi^s\psi^{s'}f_k (\phi^s\psi^{s'})\m 
\end{align*}
(it is well defined because $\phi^s$ and $\psi^{s'}$ commute with 
$\rho_0(\gamma)=(f_1^2\circ \dots \circ f_k^2)\m$ and $\rho_0(p(\delta))=f_k^2 \circ f_{k-1}^{-2}\circ\dots\circ f_1^{-2}$ respectively).

The assumptions of Lemma \ref{lem_strategie} hold:
$\calr$ is a Baire space, and the irreducibility follows from the fact that the maps $s\mapsto \phi_s$ and $s'\mapsto \phi_{s'}$
are polynomials in the variables $s^{\pm1},s'^{\pm 1}$.
The separation property follows from Lemma \ref{lem_p3}
together with the fact that $\rho_{\lambda^N,\mu^N}=\rho_0\circ p\circ \tau^N$, and that $\rho_0$ is injective.

\subsection{Embeddings in $\Diff(\bfk,0)$}\label{par:embeddings-other-fields-dense}
The proofs just given  provide the following statement.

\begin{named}{Theorem B}
Let $(\bfk, \vert \cdot \vert)$ be a non-discrete and complete valued field.
\begin{itemize}
\item[(1)] Let $\Gamma$ be the fundamental group of a closed orientable surface, or a closed non-orientable surface
of genus $\geq 4$. 
 Then, there is an embedding of $\Gamma$  into $\Diff(\bfk,0)$.
 
\item[(2)] Let $F\subset \Diff(\bfk,0)$ be a free group of rank $2$, generated by two germs $f$ and $g$ with 
$\vert f'(0)\vert >1$ and $\vert g'(0)\vert > 1$. Then, there is an embedding of $\Gamma_2$, the fundamental
group of a closed, orientable surface of genus $2$, into $\Diff(\bfk,0)$ whose image contains $F$.
\end{itemize}
\end{named}

\begin{proof}
For the first assertion, we just have to replace $\R$ by $\bfk$ in the proofs of Theorem~\ref{thm_g2} and~\ref{thm_Ng}. The parameter space is $\calr = (\bfk^*)^3$ or $\bfk^*$ or $(\bfk^*)^2$, and it is
a Baire space because $(\bfk, \vert \cdot \vert)$ is complete.

For
the second assertion, we start with a representation $\rho_0$ 
in Equation~\eqref{eq_rho0} whose image is equal to $F$.  Remark~\ref{rem_g2} shows that all the injective 
morphisms $\Phi_s$ that we get satisfy also $\Phi_s(\Gamma_2)\supset F$.
\end{proof}


\medskip

\begin{center}
{\bf{ -- Part II. --}} 
\end{center}

\section{The final topology on germs of diffeomorphisms}\label{sec-appendix}

Let $(\bfk, \vert \cdot \vert)$ be a complete field. 
This section introduces a new topology on $\bfk\{z\}$ and $\Diff(\bfk,0)$, which will be used in our second proof of Theorem~A. 
The reader may very well skip this section on a first reading.


\subsection{The final topology over the complex numbers.}\label{par:finalC}

Until Section \ref{sec_topo_finale_generale}, we focus on the case $\bfk=\C$. 
Let $r$ be a positive real number. Consider the subalgebra $\DA_r$ of $\C\{z\}$ consisting of those  power series $f(z)=\sum_n a_n z^n$ which converge on the open unit disk $\disk_r$ (i.e. $\rad(f)\geq r$) and extend continuously to the closed unit disk $\overline{\disk}_r$. When endowed with the norm 
\begin{equation}
\norm{f}_{\DA_r}=\max_{z\in \overline{\disk}_r}\vert f(z)\vert,
\end{equation}
$\DA_r$ is a Banach algebra. 
If $s<r$, the restriction of functions $f\in \DA_r$  to the smaller disk ${\overline{\disk}}_s$ determines a $1$-Lipschitz embedding $\DA_r\to \DA_s$.  

The space $\C\{z\}$ is the union of the   algebras $\DA_r$ and can be thus endowed with the final topology associated to the colimit
\begin{equation}\label{eq-colimit}
\C\{z\}=\varinjlim\DA_r.
\end{equation}
This means that a subset $\U\subset\C\{z\}$ is open if its intersection with $\DA_r$ is open for every $r>0$. Equivalently, a map $\varphi\colon \C\{ z\}\to X$ to a topological space is continuous if and only if its composition with the embedding $\DA_r\to \C\{z\}$ is continuous for all $r$.
 Unless we say it explicitly, open sets, neighborhoods, and continuous maps refer, from now on, to this topology.
A word of warning: for $r>s$, the inclusion $\DA_r\to \DA_s$ is not a homeomorphism to its image, and neither is the inclusion $\DA_r\to \C\{z\}$.

The goal of this section is to obtain several basic properties of this topology. For instance, we are going to prove that 
there is a filtration of $\C\{z\}$ by compact subsets $\C_c\{z\}$ so that the continuity can be checked in restriction to each $\C_c\{z\}$.

\begin{rem}\label{rem:Montel-compact} If $s<r$, the homomorphism $\DA_r\to \DA_s$ is compact: 
by Montel theorem, the ball of radius $1$ in $\DA_r$ is mapped into a compact subset $K_1$ of $\DA_s$. 

Let $K\subset \DA_r$ be a bounded subset. Then, the closure $cl_s(K)$ of (the image of) $K$ in $\DA_s$ 
is compact. If $t\leq s$, the image of $cl_s(K)$ in $\DA_t$ is compact, hence closed; this implies that  
$cl_s(K)=cl_t(K)$ in $\C\{z\}$. Thus, the closure $\ol K$ of $K$ in $\C\{z\}$ coincides with the
closure $cl_s(K)$ of $K$ in $\DA_s$ for any $s>r$. As a consequence, $\ol K$ is compact. \end{rem}

We denote by $B_{\DA_r}(\eps)$ the open ball centred at $0$ and of radius $\epsilon$ in $\DA_{r}$, which we also view as a subset of $\C\{z\}$. 

\begin{lem}\label{lem-basis}
A subset $\U$ of $\C\{ z \}$ is a neighborhood of $0$  if and only if there are decreasing sequences $(r_n)$ and $(\epsilon_n)$ tending to $0$ 
such that $\U$ contains the set
$$\mathcal{B}=\bigcup_n\sum_{j\geq 1}^n B_{\DA_{r_j}}(\epsilon_j).$$
\end{lem}

Lemma \ref{lem-basis} shows that the topology defined in this section is the same as the topology introduced by Leslie in \cite{Leslie}, except that we consider germs of analytic functions at the origin in $\C$ instead of real analytic functions on a compact analytic manifold. 

\begin{proof}
First we argue that any set $\mathcal{B}$ as in the statement of Lemma \ref{lem-basis} is a neighborhood of $0$ in $\C\{z\}$. To do so we need to check that $\mathcal{B}\cap\DA_r$ contains a neighborhood of $0$ for all $r$. The sum $\sum_{j=1}^n B_{\DA_{r_j}}(\epsilon_j)$ is a subset of $\C\{ z \}$ which is contained in $\DA_{r_n}$. It is open in $\DA_{r_n}$ because one of the summands, namely $B_{\DA_{r_n}}(\epsilon_n)$, is itself open. Now, the continuity of the inclusion $\DA_r\to\DA_{r_n}$ for $r_n<r$ implies that $\sum_{j=1}^n B_{\DA_{r_j}}(\epsilon_j)\cap\DA_r$ is also open in $\DA_r$. Since $\sum_{j=1}^n B_{\DA_{r_j}}(\epsilon_j)\cap\DA_r$ is contained in $\mathcal{B}\cap\DA_r$, the latter is a neighborhood of $0$, as we needed to prove.

Suppose now that $\U$ is a neighborhood of the origin in $\C\{z\}$, and fix a decreasing sequence $(r_n)$ tending to $0$. For each $n\geq 1$, set $\U_n=\U \cap \DA_{r_n}$. 

We first claim that there is a ball $B_1$ in $\DA_{r_1}$ such that $\ol B_1\subset \U$.
Since $\U_2$ is open in $\DA_2$, consider $\eps>0$ such that $B_{r_2}(\eps,0)\subset \U_2$.
Now let $B_1=B_{\DA_{r_1}}(\eps/2,0)$. Then for all $\eta>0$,
$\ol B_1\subset B_1+B_{\DA_{r_2}}(\eta)$ so taking $\eta=\eps/2$, we get $\ol B_1\subset B_{r_2}(\eps/2,0)+B_{r_2}(\eps/2,0)\subset B_{\DA_{r_2}}(\eps)\subset \U$, which proves our claim.

We now construct by induction open balls $B_n\subset \DA_{r_n}$ such that 
for all $n$, $\ol B_1+\dots +\ol B_n\subset \U$.
Given such a set of balls $B_1$, $\dots$, $B_n$, the set $K=\ol B_1+\dots +\ol B_n$ provides a compact subset of $\DA_{r_{n+2}}$ contained in $\U_{n+2}$.
Let $\epsilon$ be the distance from $K$ to the complement of $\U_{n+2}$ in $\DA_{r_{n+2}}$; by compactness, $\eps>0$,
and $K+B_{\DA_{r_{n+2}}}(\eps/2)\subset U$.
We then define $B_{n+1}=B_{\DA_{r_{n+1}}}(\eps/4)$. Then $K+\ol B_{n+1}\subset K+ B_{\DA_{r_{n+2}}}(\eps/2,0)\subset \U$. 
This concludes the induction step and the proof.
\end{proof}

\subsection{Coefficient functions}
Recall that the coefficients of $f\in\DA_r$ can be computed via the Cauchy integral formula:
\begin{equation}
A_n(f)=\frac 1{2\pi i}\int_{\{\vert z\vert=r\}}\frac{f(z)}{z^{n+1}}dz.
\end{equation}
This implies that the linear form $A_n$ is continuous on each algebra $\DA_r$ with operator norm $\norm{ A_n}_{\DA_r^*}\le r^{-(n+1)}$, i.e.\ $|A_n(f)|\leq r^{-(n+1)} \norm{f}_{\DA_r}$ for all $f\in \DA_r$. Since the maps $A_n$ separate points in $\C\{z\}$, we obtain:

\begin{lem}\label{lem:coeff-continuous}
  For each $n\geq 0$, the map $A_n:\C\{z\}\to \C$ is  continuous. The topological space  $\C\{z\}$ is Hausdorff.
\end{lem}
More generally, we have:
\begin{lem}\label{lem:continuity-of-coefficients}
If $\sum_n \theta_nz^n$ is a power series with infinite convergence radius, then the quantity 
\begin{equation}\label{eq:Theta} 
\Theta(f)=\sum_n \theta_n \vert A_n(f)\vert
\end{equation}
is well defined for every $f\in\C\{z\}$ and the function $\Theta\colon \C\{z\}\to \R_+$ is continuous.
\end{lem}

\begin{proof}
  The estimate $\Vert A_n\Vert_{{\DA_r^*}}\le\frac 1{2\pi} r^{-(n+1)}$
  implies that the map
\begin{equation}
\DA_r\to\C,\ \ f\mapsto\sum_n\theta_n\vert A_n(f)\vert
\end{equation}
is continuous for any power series $\sum_n\theta_n z^n$ with
convergence radius greater than $\frac 1r$. By definition of the
topology on $\C\{z\}$ we get that this map is continuous on the whole
space if the power series in question has infinite convergence radius.
\end{proof}

\subsection{Another filtration}

We now introduce another filtration of $\C\{z\}$.
If $c$ is any positive real number,  we define
\begin{equation}
\C_c\{ z\} = \left\{f\in\C\{z\}\text{ with }\vert A_n(f)\vert \leq c^{n+1}\text{ for all }n  \right\}.
\end{equation}
Then  $\C_c\{z\}\subset \C_{c'}\{z\}$ for $c\leq c'$, and  $\C\{z\}$ is the increasing union
of all $\C_c\{z\}$.

\begin{lem}\label{Cc set is compact}
  $\C_c\{z\}$ is compact, and contained in $\DA_r$ for all $r<c\m$.
Every compact subset $\Lambda\subset \C\{z\}$ is contained in some $\C_{c}\{z\}$. 
\end{lem}

By compactness, the topology on $\C_{c}\{z\}$ induced by $\DA_{r}$ and by $\C\{z\}$ agree.

\begin{proof} From Lemma~\ref{lem:coeff-continuous}, we deduce that $\C_c\{z\}$ is closed in $\C\{z\}$.
 If $f\in\C_c\{z\}$ and $r<c\m$, then
\begin{equation}
\norm{f}_{\DA_r}\leq \sum_n c^{n+1} r^n\leq  \frac{c}{1-cr}
\end{equation}
This means that $\C_c\{z\}$ is a bounded subset in $\DA_r$. 
Since the inclusion $\DA_r\to \DA_s$ is compact for $r>s$,   $\C_c\{z\}$ has compact closure in $\DA_s$, hence in $\C\{z\}$.
Since $\C_c\{z\}$ is closed, it is compact.

To prove the second assertion, assume by contradiction
that there is a compact subset $\Lambda\subset\C\{z\}$ such that for every integer $m>0$ there exists $f_m\in \Lambda\setminus \C_m\{z\}$. 
By definition, there is an index $n_m\geq 0$ with $\vert A_{n_m}(f_m)\vert > m^{n_m+1}$. By Lemma \ref{lem:continuity-of-coefficients}, each individual coefficient is continuous and thus bounded on the compact $\Lambda$. It follows that $n_m$ goes to $+\infty$ as $m$ does. We can thus assume, passing to a subsequence if necessary, that the $n_m$'s are pairwise distinct. 

Set $\theta_{n_m}=(\frac{1}{m})^{n_m}$, and $\theta_n=0$ if $n$ is not one of the indices $n_m$. Then $\theta_n^{1/n}$ converges towards $0$ as $n$ goes to $+\infty$, meaning that the power series $\sum_n\theta_nz^n$ has infinite convergence radius. By Lemma \ref{lem:continuity-of-coefficients}, the map 
$f\mapsto\Theta(f)=\sum_n\theta_n\vert A_n(f)\vert$ is continuous on $\C\{z\}$ and thus bounded on our compact set $\Lambda$. On the other hand we have
\[
\Theta(f_{m})\geq \theta_{n_m}\vert A_{n_m}(f_m)\vert\geq m.
\]
This yields the desired contradiction.
\end{proof}

\begin{rem}
Given $r>0$, introduce 
\begin{equation}
\mathcal B_r=\left\{ f \in \C\{z\} \; |\;\rad(f)\geq r,\ \sup_{\disk_r}|f| \leq \frac{1}{r} \right\}.
\end{equation}
This is the closure in $\C\{z\}$ of a ball in $\A_r$ and is therefore compact (Remark~\ref{rem:Montel-compact}).
There are functions $c_1$, $c_2$, $r_1$, and $r_2:\R_+^*\to \R_+^*$ such that
$$\C_{c_1(r)}\{z\}\subset\mathcal B_r\subset \C_{c_2(r)}\{z\} \; \; {\text{and}} \; \;
\mathcal B_{r_1(c)} \subset\C_c\{z\}\subset \mathcal B_{r_2(c)}.$$
It follows that one could equivalently state the results of this section in terms of the filtration $(\mathcal B_r)_{r>0}$
instead of $(\C_c\{z\})_{c>0}$.
\end{rem}

The following corollary allows us to view the final topology on $\C\{z\}$ as the weak topology associated to the filtration by the compact sets $\C_c\{z\}$.

\begin{cor}\label{cor_filtration}
A subset $F\subset \C\{z\}$ is closed if and only if for all $c>0$, $F\cap \C_{c}\{z\}$ is closed.
A map $F:\C\{z\}\to X$ to a topological space is continuous if and only if
its restriction to $\C_{c}\{z\}$ is continuous for all $c>0$.
\end{cor}

\begin{proof}
Clearly, it suffices to prove the first assertion.
If $F$ is closed, so is $F\cap \C_{c}\{z\}$.
Assume conversely that $F\cap \C_{c}\{z\}$ is closed for all $c>0$, and let us prove that $F$ is closed.
By definition of the final topology, we need to prove that given $r>0$, its preimage $j_r\m(F)$ under the inclusion $j_r:\DA_r\to \C\{z\}$ is closed.
It suffices to prove that for any $R>0$, its intersection with the ball $B_{\DA_r}(R)$ is closed in $\DA_r$.
Since $B_{\DA_r}(R)$ has compact closure, there exists $c>0$ such that $B_{\DA_r}(R)\subset \C_c\{z\}$.
Since $F\cap \C_{c}\{z\}$ is closed, $B_{\DA_r}(R)\cap j_r\m(F)= B_{\DA_r}(R)\cap j_r\m(F\cap \C_{c}\{z\})$ is a closed subset of $\DA_r$ which concludes the proof.
\end{proof}

Although one can show that the topology on $\C\{z\}$ is not metrizable, each space $\C_c\{z\}$ is a metric space.
Being compact, the topology on $\C_c\{z\}$ can be described in many equivalent ways:

\begin{pro}\label{prop_eqvce_CV}
Let $c$ be a positive real number. Let   $(f_m)$ be a sequence in $\C_c\{z\}$ and let $f_\infty$ be an element of $\C_c\{z\}$. The following are equivalent:
\begin{enumerate}
\item $(f_m)$ converges to $f_\infty$ in $\C_c\{z\}$;
\item for some (any) $r<c\m$, $(f_m)$ converges uniformly toward $f_\infty$ on $\ol\disk_{r}$;
\item $(f_m)$ converges toward $f_\infty$  uniformly on every compact subset of $\disk_{c^{-1}}$;
\item for every index $n$, $A_n(f_m)$ converges toward $A_n(f_\infty)$.  
\end{enumerate}
\end{pro}

\begin{proof}
  As seen before, $\C_c\{z\}$ is contained in $\DA_r$ for all $r<c\m$ and the topology
induced by $\norm{\cdot}_{\DA_r}$ agrees with the topology induced by $\C_c\{z\}$.
This proves the equivalence of the first three assertions.

To prove the equivalence with the last assertion,
consider the map $\Phi:\C_c\{z\}\to [0,1]^\N$ defined
by $\Phi(f)=(\frac{A_n(f)}{c^{n+1}})_{n\in\N}$, where $[0,1]^\N$ is endowed with the product topology.
This map being continuous and injective, it is a homeomorphism to its image,
and the result follows.
\end{proof}




\subsection{The final topology on a  field with an absolute value}\label{sec_topo_finale_generale}
In this section, we explain that the final topology induced by the filtration $\C_c\{z\}$ makes sense for every 
field $\bfk$ with an absolute value $\vert \cdot \vert$; 
but the results based on Montel theorem (Remark~\ref{rem:Montel-compact}) may fail for fields $\bfk\neq \C$. 

Let $\bfk$ be a complete field $\bfk$ with an absolute value $|\cdot|:\bfk\to\R_+$.  
By Ostrowski's Theorem, $\bfk$ is either $\R$ or $\C$, or the absolute value is non-archimedean: $|x+y|\leq \max(|x|,|y|)$
for all $x,y\in \bfk$. The algebra $\bfk\{z\}$  of convergent power series is filtrated by the family of subsets
\begin{equation}
\bfk_c\{z\}=\left\{f\in \bfk\{z\}\text{ with }|A_n(f)|\leq c^{n+1}\text{ for all $n$}\right\}
\end{equation}
for $c>0$.
We endow $\bfk_c\{z\}$ with the product topology, via the embedding $f\in \bfk_c\{z\}\to (A_n(f))_n\in \bfk^{\N}$:
 a sequence $(f_k)_{k\in \N}$ of elements of $\bfk_c\{z\}$ converges to $f_\infty\in \bfk_c\{z\}$
if and only if $A_n(f_k)\to A_n(f_\infty)$ for all $n$.
For $c\leq c'$, $\bfk_{c}\{z\}$ is closed in $\bfk_{c'}\{z\}$ and the inclusion is a homeomorphism to its image.
We then endow $\bfk\{z\}$ with the topology associated to this filtration:  a subset $F\subset \bfk\{z\}$ is closed if
and only if $F\cap\bfk_c\{z\}$ is closed in $\bfk_c\{z\}$.
Equivalently, a map $\phi:\bfk\{z\}\to X$ to a topological space is continuous if and only if its restriction to $\bfk_c\{z\}$ 
is continuous for every $c>0$.
By construction, the maps $f\mapsto A_n(f)$ are continuous on $\bfk\{z\}$.
Proposition \ref{prop_eqvce_CV} shows that, when $\bfk=\C$, this topology agrees with the final topology defined in Section~\ref{par:finalC}.

If $\bfk$ is locally compact, each $\bfk_c\{z\}$ is compact.
In general, since $\bfk_{c}\{z\}$ is a countable product of complete metric spaces, we get:
\begin{pro}\label{baire_space} If $\bfk$ is a complete field, then
$\bfk_{c}\{z\}$ is a metrizable complete space. In particular, it is a Baire space.
\end{pro}

On the other hand, $\bfk\{z\}$ is not a Baire space since it is a countable union of $\bfk_c\{z\}$, each of which is closed and has an empty interior.

\subsection{The topological group of germs of diffeomorphisms}\label{sec_topological_group_of_germs}
 Any $f\in \Diff(\bfk,0)$ can  be written as $f=\lambda(z+z^2\tilde f)$ for some $\tilde f\in \bfk\{z\}$ or equivalently as
\begin{equation}
f=\lambda(z+\tilde a_2z^2 +\dots +\tilde a_k z^k+\dots)
\end{equation}
for some $\lambda\in \bfk^*$ and $\tilde a_n\in \bfk$.
Thus, we define the maps  $\tilde A_n:\Diff(\bfk,0)\to \bfk$ by
\begin{equation}
\tilde A_n(f)=A_n(f)/A_1(f)=\tilde a_n.
\end{equation}

Given two real numbers $c>0$ and $\lambda_0>1$, we define the two subsets 
\begin{align}
\Diff_{c}(\bfk,0)&=  
\left\{f\in \Diff(\bfk,0)\; ; \;  |\tilde A_n (f)|\leq c^{n-1}  \text{ for all $n$}\right\}\label{eq_Diffc}
\intertext{and}
\Diff_{\lambda_0,c}(\bfk,0)&=  
\left\{ f\in \Diff_c(\bfk,0)\; ; \;   \frac{1}{\lambda_0}\leq |A_1(f)|\leq \lambda_0\right\}
\end{align}

Observe that if we denote by $m_\alpha\colon z\mapsto \alpha z$ the multiplication by some scalar $\alpha\in\bfk^*$ then we have
\begin{align} \label{eq:homothetie}
m_\alpha \Diff_{c}(\bfk,0) m_\alpha^{-1}&=\Diff_{c\alpha}(\bfk,0)
\intertext{and}
m_\alpha \Diff_{\lambda_0,c}(\bfk,0) m_\alpha^{-1}&=\Diff_{\lambda_0,c\alpha}(\bfk,0)  
\end{align}
 
\begin{lem}
A map $\phi:\Diff(\bfk,0)\to X$ to a topological space is continuous
if and only if it is continuous in restriction to $\Diff_{c}(\bfk,0)$ (or equivalently to $\Diff_{\lambda_0,c}(\bfk,0)$)  for every $c>0$ and $\lambda_0>1$.
\end{lem}
\begin{proof}
It suffices to check the continuity of $\phi$ on the open set $U_{\lambda_0}=\{f\; ; \; \frac{1}{\lambda_0}<|A_1(f)|<\lambda_0\}$ for all $\lambda_0>1$.
By definition of the final topology, it suffices to check its continuity on $U_{\lambda_0}\cap \bfk_{c}\{z\}$ for every $c>1$.
But $U_{\lambda_0}\cap \bfk_{c}\{z\}$ is a subset of $\Diff_{\lambda_0,c'}(\bfk,0) $ as soon as $c'\geq \max(\lambda_0,c^3)$; since we know that $\phi$ is continuous on $\Diff_{\lambda_0,c'}(\bfk,0)$, this proves the lemma.
\end{proof}

\begin{pro}\label{prop-baire space} If $\bfk$ is a complete field, then
$\Diff_{\lambda_0,c}(\bfk,0)$ and $\Diff_c(\bfk,0)$ are complete metric spaces. In particular, they are Baire spaces.
\end{pro}

\begin{proof}
By definition, $\Diff_{c,\lambda_0}(\bfk,0)$ is homeomorphic to  a countable product of closed subsets of $\bfk$; so, $\bfk$ being complete, its topology is induced by a complete metric.
Since $\bfk^*$ is homeomorphic to the closed subset  $\{(x,y)|xy-1\}\subset \bfk^2$,
the same argument applies to $\Diff_{c}(\bfk,0)$.
\end{proof}

\begin{thm}\label{thm:topological-group}
Let $(\bfk, \vert \cdot \vert)$ be a field with a complete absolute value. 
With the final topology, $\Diff(\bfk,0)$ is a topological group. 
\end{thm}

\begin{lem} \label{lem_emboite}
For every real number  $c>1$, there exists a real number $c'>1$ such that the following holds:
if $f$ and $g$ are in $\Diff_{c,c}(\bfk,0)$, then $f\circ g$ and $f\m$ lie in $\Diff_{c',c'}(\bfk,0)$.
\end{lem}

\begin{proof}
  Let $f=\lambda(z+\sum_{n\geq 2} \tilde a_nz^n)$,  $g=\mu(z+\sum_{n\geq 2}\tilde b_nz^n)$ with $|\tilde a_n|,|\tilde b_n|\leq c^{n-1}$
and $|\lambda|$, $|\mu|$, $|\lambda|\m$, $|\nu|\m\leq c$.
Let $F=c(z+\sum_{n\geq 2}c^{n-1}z^n)=\frac{cz}{1-cz}\in \R\{z\}$ so that the absolute value of the coefficients of
$f$ and $g$ are bounded by the coefficients of $F$.
Then the absolute value of the coefficients of $f\circ g= \sum_{n\geq 1} a_n(\sum_{m\geq 1} b_mz^m)^n$
are bounded by the coefficients of $F\circ F=\frac{cz(1-cz)}{1-cz-c^2z}$.
Since $F\circ F$ has positive convergence radius, there exists $c'\geq c^2$ 
such that $\tilde A_n(F\circ F)\leq c'^{n-1}$ for all $n\geq 2$. 
The first assertion follows.

We now prove the second assertion. Let $f=\lambda(z+\sum_{n\geq 2} \tilde a_nz^n)$, and let $f\m=\lambda\m(z+\sum_{n\geq 2}\tilde  b_nz^n)$. 
The inversion formula from Section~\ref{par:formal-diff-inversion} gives 
 \begin{eqnarray*}
\vert\tilde  b_n \vert  & \leq & 
 \frac{|\lambda| }{\vert \lambda \vert^n}
 \sum_{k_1, k_2, \ldots} \frac{(n+1)\cdots(n-1+k_1+k_2+\ldots)}{k_1! \, k_2!\, \cdots} \cdot |\tilde a_2| ^{k_1}|\tilde a_3|^{k_2}\cdots \\
  & \leq &  c^{n-1} 
\sum_{k_1, k_2, \ldots} \frac{(n+1)\cdots(n-1+k_1+k_2+\ldots)}{k_1! \, k_2!\, \cdots} (c)^{k_1}(c)^{2k_2}\cdots \\
  & = &  c^{2n-2} 
\sum_{k_1, k_2, \ldots} \frac{(n+1)\cdots(n-1+k_1+k_2+\ldots)}{k_1! \, k_2!\, \cdots} .
\end{eqnarray*}
Thus, we have to bound the quantity
 \[
K_n:= \sum_{k_1, k_2, \ldots}\frac{(n+1)\cdots(n-1+k_1+k_2+\ldots)}{k_1! \, k_2!\, \cdots}.
\]
But the numbers $K_n$ are the coefficients of the power series expansion of the reciprocal diffeomorphism $g^{-1}$ of
\[
g(z)=z-z^2-z^3-z^4\dots=z\left(2-\frac{1}{1-z}\right)=\frac{z-2z^2}{1-z}.
\]
In close form, we obtain
\[
g^{-1}(y)=\frac{(1+y)}{4} - \frac{1}{4}\sqrt{1-6y+y^2}.
\]
Since $g\m$ has positive convergence radius, there exists $c_0$ such that for all $n\geq 2$, $K_n\leq c_0^{n-1}$
hence $|\tilde b_n|\leq (c_0c^2)^{n-1}$ and the result follows.
\end{proof}

\begin{proof}[Proof of Theorem~\ref{thm:topological-group}]
By definition of the topology, given $c>0$, one only needs to check the continuity of the group laws in restriction to
$\Diff_{c,c}(\bfk,0)$.

Since $A_n(f\circ g)$ and $A_n(f\m)$ are given by polynomials in the coefficients 
$\tilde A_i(f)$, $\Tilde A_i(g)$, $A_1(f)^{\pm 1}$, $A_1(g)^{\pm 1}$,
the maps $(f,g)\mapsto A_n(f\circ g)$ and $f\mapsto A_n(f\m)$ are continuous on $\Diff_{c,c}(\bfk,0)$.
By Lemma \ref{lem_emboite} there exists $c'$ such that for all $f,g\in \Diff_{c,c}(\bfk,0)$,
$f\circ g$ and $f\m$ lie in $\Diff_{c',c'}(\bfk,0)$.
Since the topology on $\Diff_{c',c'}(\bfk,0)$ is the product topology,
the continuity of the coefficients implies the continuity of the group laws.
\end{proof}

\subsection{Other topologies}\label{par:other-topologies}
 First, we would like to point out that there are 
other reasonable and useful topologies on $\C\{z\}$, but for which the group laws are not continuous.
 This the case for the so-called {\em Takens topology} \cite{MRR, Broer-Tangerman}; this is the topology  induced by the distance
\begin{equation}
\dist(f,g)=\sup_n\vert A_n(f)-A_n(g)\vert^{1/n}.
\end{equation}
Note that in particular the convergence radius of $f-g$ is large if $f$ and $g$ are close to each other in the Taken topology, and this implies
that the right translation $R_f\colon g\mapsto g\circ f$ is not continuous if the radius of convergence of $f$ is finite. Indeed, a small perturbation $g(z)+\epsilon z$ is mapped to $R_f(g+\epsilon z)=g\circ f+\epsilon f$, and the difference $\epsilon f$ is not small in the Takens topology because its radius of convergence does not depend on $\epsilon$.

We comment now on another important topology on $\Diff(\C,0)$, but for which the Baire property fails. 
Let $\Jet_\ell(\C,0)$ be the group of $\ell$-jets of diffeomorphisms 
$a_1z+\cdots +a_\ell z^\ell \mod( z^{\ell+1})$, with $a_1\neq 0$; it can be considered as a solvable algebraic group and thus  as a solvable complex Lie group. Let 
\begin{equation}
\jj_\ell\colon \Diff(\C,0)\to \Jet_\ell(\C,0)
\end{equation}
denote the homomorphism that maps a power series $f=\sum_n a_n z^n$ to $\sum_{n=1}^\ell a_n z^n$. 
We can then define a topology on $\Diff(\C,0)$ (resp. on $\widehat{\Diff}(\C,0)$): the weakest topology for which all projections 
$\jj_\ell$ are continuous. With this topology, $\Diff(\C,0)$ is a topological group, because the projections $\jj_\ell$ are homomorphisms. Moreover, a sequence $(f_m)$ converges toward a germ of diffeomorphism $g$ if and only if the coefficients $A_n(f_m)$ converge to $A_n(g)$ for all $n$. In other words, this is the topology of simple convergence on the coefficients. In particular, $\Diff(\C,0)$ is not a closed subset of $\widehat{\Diff}(\C,0)$ for this topology. With this topology,  $\widehat{\Diff}(\C,0)$ is a Baire space, but $\Diff(\C,0)$ is not (Proposition \ref{prop-baire space} fails if $\Diff(\C,0)$ is endowed with this topology).

\subsection{Continuity in the Koenigs linearization Theorem}\label{sec_Koenigs_cv}
A contraction $f\in\Diff(\bfk,0)$ is an element with $|A_1(f)|< 1$. In this case,
 Koenigs theorem says that the unique formal 
diffeomorphism $h_f$ tangent to the identity that conjugates $f$ to the homothety $z\mapsto A_1(f)z$
has positive convergence radius. The following result shows that 
$f\mapsto h_f$ is continuous for the final topology on the set of contractions
\begin{equation}\label{eq:contractions}
\Cont(\bfk,0)=\{f\in\Diff(\bfk,0)\,|\ |A_1(f)|<1\}.
\end{equation}

\begin{thm}
\label{thm:Koenigs_continuity} 
Let $\bfk$ be a field with a complete non-trivial absolute value. 
For every germ $f\in \Cont(\bfk,0)$,
the unique formal diffeomorphism $h_f$ such that
\[
h_f(f(z))=A_1(f)\cdot h_f(z) \text{\quad and \quad} A_1(h_f)=1 
\] 
has positive convergence radius, and the map 
$$h:f\in \Cont(\bfk,0)\mapsto h_f \in \Diff(\bfk,0)$$
is continuous for the final topology.
The coefficients of $h_f$ 
are polynomial functions with integer coefficients in the variables $A_i(f)$ and $(A_1(f)^j-1)^{-1}$, for $i,j\geq 1$.
\end{thm}

When $\bfk=\C$, it is shown in \cite[Chapter 8]{Milnor:book} that  $h_f$ is convergent and its coefficients depend holomorphically on 
$f$. Theorem~\ref{thm:Koenigs_continuity} is just a variation on this classical result.

\begin{proof} We refer to \cite{Siegel} for the real and complex cases, and to \cite{Herman-Yoccoz:1983} for the non-archimedian ones.

The coefficients of $h_f$ 
can be computed inductively 
and turn out to be polynomials with integer coefficients in the variables $A_i(f)$ and $(A_1(f)^j-1)^{-1}$, for $i,j\geq 1$ (see for instance \cite[Eq 4]{Siegel}).
If $\vert A_1(f)\vert \leq \alpha$ for some $\alpha$ in the interval $[0,1[$, then 
\begin{equation}
\vert (A_1(f))^n -1\vert \geq 1-\alpha
\end{equation}
for all $n>0$. By \cite[Theorem 1]{Siegel}  
and \cite[Theorem 1]{Herman-Yoccoz:1983}  
in the archimedean and non-archimedean cases respectively, 
$h_f$ is convergent
and for all $c$, $\lambda>1$,
there exists $c'$ such that $h_f\in \Diff_{c'}(\bfk,0)$ if $f\in \Diff_{\lambda,c}(\bfk, 0)$. 

The topology on $\Diff_{c'}(\C, 0)$ is the product topology on the coefficients. Since the coefficients of $h_f$ are continuous functions of $f$,
it follows that the restriction of $f\mapsto h_f$ to $\Diff_{\lambda,c}(f)$ is continuous. By definition of the final topology,
this proves the continuity of $h$.
\end{proof}

\section{A large irreducible component of the representation variety}\label{sec_Dominique}

This section describes our second proof strategy for Theorem~A. For simplicity, we consider only 
the fundamental group of a closed orientable surface of genus~$2$, but we work over any 
field $\bfk$ with a complete absolute value $\vert \cdot \vert$.

 \subsection{An irreducible set of representations}

Using the presentation 
\begin{equation}
\Gamma_2=\grp{a,b,\ol a,\ol b\,|\ [a,b]=[\ol a,\ol b]},
\end{equation}
we get an idenfication
\begin{equation}
\Hom(\Gamma_2;\Diff(\bfk,0))=\{(f,g,\overline{f},\overline{g})\in\widehat\Diff(\bfk,0)^4 \;\vert\; [f,g]=[\overline{f},\overline{g}]\}.
\end{equation}

Let $\BX\subset\Hom(\Gamma_2,\Diff(\bfk,0))$
be the set of representations $\rho:\Gamma_2\to \Diff(\bfk,0))$
such that $\rho(a)$ is tangent to $\Id$ and $\rho(b)$ is a contraction.
As in Equation~\eqref{eq:contractions}, we denote by $\Cont(\bfk,0)$ the set of contractions.
For $c>0$, we let $\BX_c=\BX\cap \Diff_c(\C,0)$.
Set 
\begin{align}
\calr&= \Cont(\bfk,0) \times \Diff(\bfk,0) \times\Diff(\bfk,0),\\
\calr(c)&=\Cont_c(\bfk,0) \times \Diff_c(\bfk,0) \times\Diff_c(\bfk,0),
\end{align}
and denote by $\pi\colon \BX  \to  \calr $ the projection 
\begin{align}
\pi(\rho)  & = (\rho(b),\rho(\ol a),\rho(\ol b))
\label{eq:1}  
\end{align}

\begin{pro}\label{dominic's trick2}
The map $\pi$ is a  homeomorphism for the final topology,  and its inverse 
\[
\pi\m:(g,\ol f, \ol g)\mapsto (f,g,\ol f, \ol g)
\]
is a polynomial map, in the following sense:
for each $n\in \N^*$, the map $(g,\ol f, \ol g)\mapsto A_n(f)$ 
is polynomial in (finitely many of) the variables $A_k(g)$, $A_k(\ol f)$, $A_k(\ol g)$,
$A_1(g)\m$, $A_1(\ol f)\m$, $A_1(\ol g)\m$  and $(A_1(g)^k-1)^{-1}$ ($k\geq 1$).
\end{pro}

\begin{proof} The projection $\pi$ is continuous because both $\BX$ and $\calr$ come
with the topology induced by the same topology on $\Diff(\bfk,0)$.

Consider a triple $(g,\ol f, \ol g)\in \Cont(\bfk,0)\times \Diff(\bfk,0)\times \Diff(\bfk,0)$.
Since $[\ol f,\ol g]$ is tangent to the identity, the germs $g$ and $[\ol f,\ol g]\circ g$
have the same derivative $\lambda=A_1( g)$ at $0$.
Since $\vert \lambda \vert <1$, we can apply Koenigs Theorem \ref{thm:Koenigs_continuity}: 
we get two germs $h_1$ and $h_2\in \Diff(\bfk,0)$ tangent to the identity such that 
\begin{equation}
h_1\circ g \circ h_1\m=m_\lambda\text{\quad and\quad }
h_2\circ ([\ol f,\ol g]\circ{g}) \circ h_2\m=m_\lambda
\end{equation}
where $m_\lambda(z)=\lambda z$ is the multiplication by $\lambda$.
Then, the map $ f:=h_2\m\circ h_1$  conjugates $ g$ to $[\ol f,\ol g]\circ g$ so
$$ f\circ  g\circ  f\m= [\ol f,\ol g]\circ  g
\text{\quad and \quad}
[f, g]= [\ol f,\ol g].$$
This means that one can define the preimage $\pi\m(g,\ol f,\ol g)\in\Hom(\Gamma_2,\Diff(\bfk,0))$
by the 4-tuple $(f,g,\ol f,\ol g)$: the fact that $\pi\m\circ\pi=\Id_{\calr}$ follows from uniqueness in Koenigs Theorem.

The continuity of $\pi\m$ is a consequence of  the continuity of the conjugacy in Koenigs Theorem \ref{thm:Koenigs_continuity} and 
of the continuity of the map $(g,\ol f,\ol g)\mapsto [\ol f,\ol g]\circ{g}$.
The fact that $A_n(f)$ is polynomial in the given variables is
a direct consequence of the corresponding fact in Koenigs Theorem,
and the fact that group operations are polynomial mappings. \end{proof}

We denote the inverse map $\pi\m$ by $\Phi$: 
\begin{equation}
\forall s\in \calr, \quad \Phi_s=\pi^{-1}(s).
\end{equation}
Thus, if $s=(g, \ol f,\ol g)$, then $\Phi_s$ is the morphism $\Gamma_2\ra \Diff(\bfk,0)$ such that $\Phi_s(b)=g$, $\Phi_s(\ol a)=\ol f$,
$\Phi_s(\ol b)=\ol g$, and $\Phi_s(a)$ is the unique germ of diffeomorphism $f$ which is tangent to the identity and satisfies the
relation $[f,g]=[\ol f,\ol g]$.
To conclude the proof, our goal now is to prove that for every $c>0$, the family of morphisms $\Phi_s$, for  $s\in\calr(c)$, satisfies the assumptions of Lemma \ref{lem_strategie}.
Proposition~\ref{prop-baire space} shows that $\calr(c)$ is a Baire space. The following corollary proves the irreducibility of~$\calr(c)$.

\begin{cor}
For  any $w\in\Gamma_2$, denote by $\calr(c)_w\subset \calr(c)$  the set of homomorphisms in $\calr(c)$ that kill $w$.
Then either $\calr(c)_w=\calr(c)$ or $\calr(c)_w$ is a closed subset  of $\calr(c)$ with empty interior.
\end{cor}

\begin{proof} 
Since the functions $s\in \calr\mapsto A_k(\Phi_s(g))- A_k(\Id)$ are continuous, $\calr(c)_w$ is closed. 
Now, assume that $\calr(c)_w\neq \calr(c)$: there exists $k\geq 1$ and a point $s=(g_0, \ol f_0, \ol g_0)$ in $\calr(c)$ such
that $A_k(\Phi_s(w))\neq A_k(\Id)$.
 According to Proposition~\ref{dominic's trick2} the map 
 $s\mapsto A_k(\phi_s(w))- A_k(\Id)$ is a polynomial function in finitely many of 
 \begin{itemize}
\item the coefficients $A_n(g_0)$, $A_n(\ol f_0)$ and $A_n(\ol g_0)$  ($n\geq 1$),
\item the inverses $A_1(g_0)\m$, $A_1(\ol f_0)\m$, $A_1(\ol g_0)\m$  and $(A_1(g_0)^k-1)^{-1}$ ($k\geq 1$)
\end{itemize}
(note that $A_1(g_0)$, $A_1(\ol f_0)$, $A_1(\ol g_0)$  and $(A_1(g_0)^k-1)$ do not vanish on $\calr$). Our assumption says that
 this function does not vanish identically on  $\calr(c)$. 
 Assume that $\calr(c)_w$ contains a non-empty open subset $\U$, and choose a point $s=(g_1, \ol f_1, \ol g_1)$ in $\U$. 
 If $\bfk=\R$ or $\C$, we denote by $B_\bfk$ the interval $[0,1]\subset \R$;  in the non-archimedean case we set $B_\bfk=\{t\in \bfk,|t|\leq 1\}$.
 Then, we consider the convex combination 
 \begin{equation}
s_t=\left( g_t=tf_1+(1-t) f_0,\;  \ol f_t=t\ol f_1 + (1-t)\ol f_0, \ol g_t=t\ol g_1+(1-t)\ol g_0\right)
 \end{equation}
 with $t$ in $B_\bfk$.
 According to Lemma~\ref{lem_convexe} below, $g_t$, $\ol f_t$ and $\ol g_t$ are in $\calr(c')$ for some $c'\geq c$, and 
 $t\mapsto s_t$ is continuous; thus $\{ t\; ; \; s_t\in \U\}$ is an open neighborhood of $1$. 
 
 The function $t\mapsto \A_n(\phi_{s_t}(w))-A_n(\Id)$ does not vanish for $t=0$, it is the restriction of a rational function 
of the variable $t$ to the interval $[0,1]$, and it vanishes identically on the open set $\{ t\; ; \; s_t\in \U\}$. This is a contradiction, which 
shows that the interior of $\calr(c)_w$ is empty.
 \end{proof}

 Let $B_\bfk$ be the interval $[0,1]\subset \R$ if $\bfk=\R$ or $\C$, or the ball $\{t\in \bfk,|t|\leq 1\}$ in the non-archimedean case.

\begin{lem}\label{lem_convexe}
  Let $f_0\in \Diff_{c_0}(\bfk,0)$ and $f_1\in \Diff_{c_1}(\bfk,0)$,
and for $t\in \bfk$, let $f_t=(1-t)f_0+tf_1$.
Let $p\in \bfk$ be the value of $t$ (if any) such that $f'_p(0)=0$.

If $c_0\leq c_1$ and $c_0|f_0'(0)|\leq c_1|f_1'(0)|$, then 
for all $t\in B_\bfk\setminus \{p\}$, $f_t\in \Diff_{c_1}(\bfk,0)$.
\end{lem}

\begin{proof}
Denote $\lambda_0=|f_0'(0)|$ and $\lambda_1=|f'_1(0)|$.
By assumption, for all $n\geq 2$, $|A_n(f_0)|\leq \lambda_0c_0^{n-1}$
and $|A_n(f_1)|\leq \lambda_1c_1^{n-1}$.

Consider first the case   $\bfk=\R$ or $\C$.
Since $\lambda_0c_0\leq \lambda_1c_1$,
we get for all $t\in [0,1]$,
\begin{align}|A_n(f_t)| &\leq (1-t)\lambda_0 c_0^{n-1}+ t\lambda_1 c_1^{n-1}\\
&\leq (1-t)\lambda_1 c_1 c_0^{n-2}+ t\lambda_1 c_1^{n-1}\leq \lambda_1 c_1^{n-1}.\end{align}
This shows that $f_t\in \Diff_{c_1}(\bfk)$ as soon as $f_t'(0)\neq 0$.

In the non-archimedean case, one has $|t-1|\leq 1$ for $t\in B_\bfk$.
Similarly, we get
\begin{align}|A_n(f_t)| 
&\leq \max\{\, |1-t|\lambda_0 c_0^{n-1},\, |t|\lambda_1 c_1^{n-1}\,\} \\
&\leq \max\{\, \lambda_1 c_1 c_0^{n-2},\, \lambda_1 c_1^{n-1}\,\}\leq \lambda_1 c_1^{n-1}.\end{align}
This shows that $f_t\in \Diff_{c_1}(\bfk)$ as soon as $f_t'(0)\neq 0$. Then, the continuity follows 
from the continuity of the coefficients $t\mapsto A_n(f_t)$.
\end{proof}

\subsection{Separation}
To conclude the proof, we fix $c>0$ and prove that $\calr(c)$ satisfies the separation condition of Lemma \ref{lem_strategie}.
We thus fix $g\in \Gamma_2\setminus \{1\}$ and show that $\calr(c)$ contains a representation that does not kill $g$.
 Write the orientable surface group of genus $2$
 as $\Gamma_2=\langle a, b, \overline{a},\overline{b} \;\vert\; [a,b]=[\overline{a},\overline{b}] \rangle$,
 and let $p:\Gamma_2\ra \grp{a,b}$ be the morphism fixing $a$ and $b$ and sending $\ol a$ and $\ol b$ to
 $a$ and $b$ respectively.
 Let $\tau:\Gamma_2\ra\Gamma_2$ be the Dehn twist around the curve $c=[a,b]$, i.e. the automorphism that fixes $a,b$
 and sends $\ol a$ and $\ol b$ to $c\ol a c\m$ and $c\ol b b\m$ respectively. 
According to Proposition~\ref{pro_p1},
  there exists a positive integer $n_0$ such that 
$p\circ\tau^N(g)\neq 1$ for all $N\geq n_0$.  

 Apply Theorem \ref{thm_F2} to get  a pair $f_1$, $f_2$ of germs of diffeomorphisms generating a free group  $\langle f_1, f_2\rangle$ of rank $2$
and satisfying $f_1'(0)>1$ and $f_2'(0)>1$. Define a morphism $\rho:\grp{a,b}\ra \Diff(\bfk,0)$   by  $\rho(a)=[f_1,f_2]$ and $\rho(b)=f_2\m$.
Then  $\rho$ is injective, $\rho(a)$ is tangent to the identity, and $\rho(b)$ is a contraction. Set $\rho_N:=\rho\circ p\circ \tau^N$.
For $N\geq n_0$  $ \rho_N(g)\neq 1$.
Thus, $\pi(\rho_N)$ lies in $\calr\setminus \calr_g$; but it might not lie in $\calr(c)$.

Let $c_N>0$ be such that $\pi(\rho_N)\in\calr(c_N)$.
Given $\alpha\in\bfk^*$, let $\ad_\alpha$ be the inner automorphism of $\Diff(\bfk,0)$ given by
$f\mapsto m_\alpha\circ f\circ m_\alpha\m$. As noticed in Equation~\eqref{eq_Diffc}, we have $\ad_\alpha(\Diff_{c_N}(\bfk,0))=\Diff_{\alpha c_N}(\bfk,0)$. Thus, 
the representation $\rho'_N=\ad_{\alpha}\circ \rho_N$ satisfies $\pi(\rho'_N)\in \calr(c)$ if $\alpha$ is sufficiently small.
Since $\rho'_N(g)\neq 1$ this concludes that $\calr(c)$ satisfies the separation condition of Lemma~\ref{lem_strategie}.

\subsection{Conclusion} The family of representations $\Phi_s$, with $s\in \calr(c)$ satisfies the Baire property, the irreducibility property, and
the separation property of Lemma~\ref{lem_strategie}. This lemma implies that a generic element of $s\in \calr(c)$ gives an embedding
$\Phi_s\colon \Gamma_2 \to \Diff(\bfk,0)$, proving Theorem~A for the group $\Gamma_2$.

\medskip


\begin{center}
{\bf{ -- Part III. --}}
\end{center}

\section{A $p$-adic proof}\label{par:p-adic-proof}

\subsection{Free groups with integer coefficients}\label{par:free-group-Z}

A theorem of White \cite{White:1988} shows that the homeomorphisms of $\R$ defined by $f:z\mapsto z+1$ and $g:z\mapsto z^3$ generate a free group.
Conjugating the maps $f$ and $gfg\m$ by $z\mapsto \frac{1}{3z}$, as in \cite{Glass:1992}\footnote{This conjugacy is called the ``Wilson trick'' in \cite{Glass:1992}.}, one gets two formal diffeomorphisms
\begin{equation}\label{eq-free}
\begin{split}
f_0(z)&=\frac{z}{1+3z}=\sum_{n=1}^\infty(-3)^{n-1}z^n\\
g_0(z)&=\frac{z}{(1+(3z)^3)^{1/3}}=\sum_{n=0}^\infty
\left(\begin{array}{c}\frac{-1}3\\n\end{array}\right)3^{3n}z^{3n+1}
\end{split}
\end{equation}
that generate a non-abelian free group  $\langle f_0,g_0\rangle\subset\widehat\Diff(\Q,0)\subset\widehat\Diff(\bfk,0)$. 
It is remarkable that $f_0$ and $g_0$ are tangent to the identity at the origin and have integer coefficients:
 
 \begin{thm}\label{thm:thefreegroup}
The group $\widehat\Diff(\Q,0)$ contains a non-abelian free group, all of whose elements are tangent to the identity and
have integer coefficients. \end{thm}

Thus, one can produce an explicit free group in $\FD$ for every field $\bfk$ of characteristic $0$.
In characteristic $p>0$, Szegedy proved that almost every pair of elements in the Nottingham group $\widehat\Diff(\Z/p\Z,0)$ 
generates a free group \cite{Szegedy}.

\subsection{Subgroups of $\Diff(\Q_p,0)$} 
In this section, $p$ is a prime number, and $\Q_p$ is the field 
of $p$-adic numbers, with its absolute value $\vert\cdot\vert$ normalized by  $\vert p\vert=1/p$.

Let $G_p$ denote the set of elements $f=\sum_{n\geq 1} a_n z^n$ in $\Diff(\Q_p,0)$ such that 
\begin{equation}
a_n\in \Z_p \; \forall n\;\;   {\text{ and }}  \;\;\vert a_1 \vert =1.
\end{equation}
Every element $f\in G_p$ satisfies $\rad(f)\geq 1$. The ultrametric inequality and the 
Inversion formula show that $G_p$ is a subgroup of $\Diff(\Q_p,0)$. With the product topology on coefficients (as in 
Section~\ref{par:other-topologies}), it is a compact topological group, and the morphism $\jj_\ell\colon G_p\to \Jet_\ell(\Q_p,0)$ 
is continuous for every integer $\ell\geq 1$. The kernel of $\jj_\ell$ will be denoted $G_{p,\ell}$.

From Theorem~\ref{thm:thefreegroup}, we know that $G_p$ contains a free group of rank two
generated by two germs $f_0$ and $g_0$ whose coefficients are in $\Z$. 

\begin{cor}\label{cor:7.2}
Let $p$ be a prime number and $\ell$ be a positive integer. The group $ G_{p,\ell}$ contains a non-abelian free group. 
The group $G_p$ contains a free group $\grp{f,g}$ of rank $2$  such that $A_1(f)$ is a transcendental number
while $g$ is tangent to the identity up to order~$\ell$.
\end{cor}

\begin{proof}
Start with a non-abelian free group $F$ in $G_p$.
Since the group of jets $\Jet_\ell(\Q_p,0)$  is solvable, the restriction of $\jj_\ell$ to $F$ is not injective. Its kernel is 
a free group (as any subgroup of $F$), and if $\ell$ is large it is not cyclic. Thus, the kernel is
a non-abelian free group. This proves the first statement. 

Set
$\calr=\{ t\in \Z_p\; ; \; \vert t\vert =1\}$.
Now, take a pair of generators $f_0$ and $g_0$ of a free group of rank $2$ in $G_{p,\ell}$, and for $t\in \calr$ consider the 
family of representations $\rho_t\colon \F_2=\grp{a,b}\to G_p$ defined by $\rho_t(a)=m_t\circ f_0$ 
and $\rho_t(b)=g_0$  (here, as usual, $m_t(z)=tz$). 
If $w$ is an element of $\F_2$, and $n$ is a positive integer, then $A_n(\rho_t(w))$ is a polynomial function 
in $t$ and $1/t$ (see Section~\ref{par:formal-diff-inversion}). If $w\neq 1$, there is an integer $n\geq 1$ such that $A_n(\rho_1(w))\neq A_n(\Id)$. 
Thus, the set $\calr_{ w}\subset \calr$ of parameters $s$ such that $\rho_s(w)= \Id$ is finite, the union $\cup_{w\neq 1}\calr_w$
is at most countable, and there are transcendental numbers in its complement. For such a parameter $t$, 
$\rho_t$ is injective and $A_1(\rho_t(a))=t$ is transcendental. 
\end{proof}

Now, we apply the result of \cite{BGSS} described in Section~\ref{par:compact-groups} to get:

\begin{thm}\label{thm-g2-p-adic}
Let $p$ be a prime number. Let $\Gamma_2=\grp{a,b,\ol a,\ol b\,|\ [a,b]=[\ol a, \ol b]}$ be the fundamental 
group of a closed orientable surface of genus $2$. Then
\begin{enumerate}
\item For  every integer $\ell\geq 1$, the group $\Gamma_2$ embeds in the compact group $ G_{p,\ell}$.

\item There is an embedding $\rho\colon \Gamma_2\to G_p$ such that  $\rho(a)'(0)=\rho(\ol a)'(0)$ is a transcendental number
while $\rho(b)$ and $\rho(\overline{b})$ are tangent to the identity up to order~$\ell$.
\end{enumerate}
\end{thm}

\subsection{Back to complex coefficients} 
The field $\Q_p$, and thus the ring $\Z_p$, embeds (although not continuously) into $\C$; such an embedding induces 
an embedding, coefficient by coefficient, of $\Z_p[[z]]$ into $\C[[z]]$. 
Thus, the surface groups constructed in Theorem~\ref{thm-g2-p-adic} provide surface groups in $\widehat{\Diff}(\C,0)$. 
This construction does not preserve the convergence of power series, but it preserves the order of tangency to $\Id$. 
Since there are transcendental complex numbers with modulus $<1$, we obtain:

\begin{cor}\label{coro:p-adic}
Let $\ell$ be a positive integer. There is an embedding 
$
\rho\colon \Gamma_2 \to\widehat{\Diff}(\C,0)
$
such that $\vert \rho(a)'(0)\vert= \vert \rho(\ol a)'(0)\vert<1$ while $\rho(b)$ and $\rho(\overline{b})$ are tangent to the identity up to order~$\ell$.
\end{cor}

We can now prove the following version of Theorem~A. This will be our third and last proof of it.

\begin{thm}There is an embedding of $\Gamma_2$ in $\Diff(\C,0)$ such that  $\vert \rho(a)'(0)\vert= \vert \rho(\ol a)'(0)\vert<1$ while $\rho(b)$ and $\rho(\overline{b})$ are tangent to the identity up to order~$\ell$.
\end{thm}

\begin{proof}
The first step  is to choose a sequence $\Coeff=(a_1,a_2,a_3,\ldots )$ of complex numbers such that 
\begin{enumerate}
\item[(a)] the set $\{ a_1, a_2, \ldots\}$ is algebraically free: if $m\geq 1$ and $P\in \Z[x_1, \ldots, x_m]$, 
and if $P(a_1, \ldots, a_m)=0$, then $P=0$; 
\item[(b)]  $\vert a_n\vert\le 2^{-n}$ for all $n\geq 1$.
\end{enumerate}
Such a sequence exists because $\C$ is uncountable. Concrete  examples can be obtained from the Lindemann-Weierstrass theorem (see  also \cite{Waldschmidt:1990} for the constructions of von Neumann, Perron, Kneser, and Durand of uncountably many, algebraically free complex numbers). We shall consider the $a_i$ as indeterminates for the field of rational functions $\Q(a_1, a_2, \ldots)$.
Armed with such a  set we consider the following three formal diffeomorphisms
\begin{equation}
g=a_1z+\sum_{i=1}^\infty a_{3i+1}z^{i+1},\quad \bar f=z+\sum_{i=\ell}^\infty a_{3i+2}z^{i+1},\quad \bar g=z+\sum_{i=\ell}^\infty a_{3i+3}z^{i+1}.
\end{equation}
From the decay relation~(b), these three power series have a positive radius of convergence. 
Since $\vert a_1\vert \leq 1/2$, the  Koenigs linearization theorem gives a unique element $\bar f\in \Diff(\C,0)$ with $\bar f'(0)=1$ such that
\begin{equation}
fgf^{-1}=[\bar f,\bar g] g.
\end{equation}
The four elements $(f,g,\bar f,\bar g)$ determine  a representation $\phi$ of $\Gamma_2$ into $ \Diff(\C,0)$.

Let us prove that this representation is faithfull. Fix a non-trivial element $w$ of $\Gamma_2$, and
write it as a word in $a$, $b$, $\ol a$, $\ol b$ and their inverses. For every integer $n$, the coefficient 
$A_n(\phi(w))$ is a polynomial function $Q_{w,n}$ in the variables $a_n$ (for $n\geq 1$), $a_1^{-1}$, and the $(a_1^k-1)^{-1}$ (for $k\geq 1$)
with integer coefficients.

Now, take a faithful representation $\rho\colon \Gamma_2\to \widehat{\Diff}(\C,0)$ that satisfies the conclusion of Corollary \ref{coro:p-adic}.  
There is an integer $n\geq 1$ such that $A_n(\rho(w))\neq A_n(\Id)$. This implies that $Q_{w,n}\neq A_n(\Id)$ when we
specialize the indeterminates $a_i$ to the coefficients of the generators $\rho(\ol a)$, $\rho(b)$, and $\rho(\ol b)$. Since 
$Q_{w,n}\neq A_n(\Id)$, $\phi(w)\neq \Id$ and $\phi$ is the identity.
\end{proof}

\medskip


\begin{center}
{\bf{ -- Part IV. --}}
\end{center}

\section{Complements  and open questions}\label{sec_questions}

\subsection{Takens' theorem and smooth diffeomorphisms} 

To conclude this chapter, we mention the following result which allows to realize \emph{any} faithful representation
of a surface group in the group of formal germs as a group of $C^\infty$ germs. Note that the $p$-adic method 
provides many embeddings of surface groups in $\widehat{\Diff}(\R,0)$ (see Corollary~\ref{coro:p-adic}).

Recall that $\Gamma_g$ denotes the fundamental group of the closed orientable surface of genus $g$.

\begin{named}{Theorem C}
Let $\hat{\rho}\colon \Gamma_g\to \widehat{\Diff}(\R,0)$ be a faithful representation of the surface
group $\Gamma_g$ in the group of formal diffeomorphisms in one real variable. Then, there exists 
a faithful representation $\rho\colon \Gamma_g\to \Diff^\infty(\R,0)$ into the group of germs of
${\mathcal{C}}^{\infty}$-diffeomorphisms such that the Taylor expansion of $\rho(w)$ coincides
with $\hat{\rho}(w)$ for every $w\in \Gamma_g$. 
\end{named}

The proof will be  a consequence of the following result (this theorem is easily derived from the Sternberg linearization 
theorem and Theorem 2 of \cite{Takens:1973}): 

\begin{thm}[Sternberg \cite{Sternberg:1958}, Takens, \cite{Takens:1973}]\label{thm:Sternberg-Takens}
Let $f, g\colon (\R,0)\to (\R,0)$ be two germs of ${\mathcal{C}}^\infty$-diffeo\-morphisms, and
let $\hat{f}$ and $\hat{g}$ denote their Taylor expansions. 
Suppose that $f$ is not flat to the identity, that is $\hat{f}\neq {\mathrm{Id}}$. Then, if 
$\hat{f}$ and $\hat{g}$ are conjugate by a formal diffeomorphism $\hat{h}$, there exists
a germ of ${\mathcal{C}}^\infty$-diffeomorphism $h\colon (\R,0)\to (\R,0)$ such that 
\begin{itemize}
\item the Taylor expansion of $h$ coincides with $\hat{h}$;
\item $h$ conjugates $f$ to $g$. 
\end{itemize}
\end{thm}

\begin{proof}[Proof of Theorem~C]
Denote by $\hat{a}_i$, $\hat{b}_i$, $1\leq i\leq g$ the images of the standard 
generators of $\Gamma_g$ by the representation $\rho$; they satisfy the relation 
\begin{equation}
\hat{a}_1\circ \hat{b}_1\circ \hat{a}_1^{-1}=(\prod_{j=2}^g [\hat{a}_j,\hat{b}_j])\circ \hat{b}_1.
\end{equation}
By the  theorem of Borel and Peano, one can find germs of diffeomorphisms  $b_1$ and $a_j$, $b_j$, $j\geq 2$, whose respective
Taylor expansions coincide with $\hat{b}_1$,  $\hat{a}_j$, and $\hat{b}_j$ respectively. Then, Theorem~\ref{thm:Sternberg-Takens}
provides a germ of diffeomorphism $a_1$ such that $a_1\circ b_1\circ a_1^{-1}= (\prod_{j=2}^g [a_j,a_j])\circ b_1$. Thus, one 
gets a representation $\rho$ of $\Gamma_2$ into $\Diff^\infty(\R,0)$ with Taylor expansion equal
to $\hat{\rho}$. Since the initial representation $\hat{\rho}$ is injective, so is $\rho$. 
\end{proof}

\subsection{Conjugacy classes}
Two subgroups $\Gamma_1$ and $\Gamma_2$  of $\Diff(\C,0)$ are {\bf{topologically 
conjugate}} if there is a germ of homeomorphism $\varphi\colon (\C,0)\to (\C,0)$ 
such that $\varphi\circ \Gamma_1\circ \varphi^{-1}=\Gamma_2$, and are {\bf{formally 
conjugate}} if there is a formal diffeomorphism $\hat{\varphi}$ such that 
$\hat{\varphi}\circ \Gamma_1\circ \hat{\varphi}^{-1}=\Gamma_2$. 
A germ of homeomorphism $\varphi$ is {\bf{anti-holomorphic}} if its complex conjugate $z\mapsto \overline{\varphi(z)}$
is holomorphic.
\begin{thm}[Nakai, Cerveau-Moussu] 
Let $\Gamma_1$ and $\Gamma_2$ be two subgroups of $\Diff(\C,0)$ which are not solvable. 
\begin{enumerate}
\item If $\varphi$ is a local homeomorphism that conjugates $\Gamma_1$ to $\Gamma_2$, then 
$\varphi$ is holomorphic, or anti-holomorphic.
\item If $\hat{\varphi}$ is a formal conjugacy between $\Gamma_1$ and $\Gamma_2$, then $\hat{\varphi}$
converges and is therefore a holomorphic conjugacy.
\end{enumerate}
\end{thm}
Thus, (the images of) two embeddings of $\Gamma_g$ in $\Diff(\C,0)$ are topologically or formally conjugate if 
and only if they are analytically conjugate. 

\subsection{Two questions}

\subsubsection{} It would be interesting to exhibit an embedding $\alpha$ of the group $\Gamma_g$, $g\geq 2$, into the
group of analytic diffeomorphisms of the circle $\R/Z$ fixing the origin $o\in \R/\Z$. If such an embedding 
exists, the suspension of this representation $\alpha$ gives a compact manifold $M_\alpha$ of dimension $3$ that
fibers over $\Sigma_g$, together with a foliation ${\mathcal{F}}_\alpha$ of co-dimension $1$ which is transverse to the fibration 
$\pi\colon M_\alpha\to \Sigma_g$ and whose monodromy is given by $\tau$. The fixed point gives a compact
leaf of ${\mathcal{F}}_\alpha$ with holonomy given by the same representation $\tau$.

\medskip

{\bf{Question.-- }} Does there exist an embedding of $\Gamma_2$ into the group of analytic
diffeomorphisms of the circle fixing the origin ?

\medskip

This question was the original motivation of  Cerveau and Ghys when they asked for a proof 
of Theorem~A (see \cite{Cerveau:Survey2013}). 

\begin{rem}
According to Theorem~\ref{thm-g2-p-adic}, there is an embedding $\rho$ of $\Gamma_2$ in $\Diff(\Q_p,0)$
such that  $\rho(a)'(0)$ and $\rho({\overline{a}})'(0)$ have modulus $1$ while $ \rho(b)$ and $\rho({\overline{b}})$
are tangent to the identity. Conjugate $\rho$ by the homothety $a\mapsto p^N z$ for some positive integer $N$. 
If $N$ is large enough, the coefficients $a_n$, $n\geq 2$, of all elements of $\rho(\Gamma_2)$ have norm $<1$, 
and the ultrametric inequality shows that $\rho(\Gamma_2)$ preserves the open disks $\{z\in \C_p\; ; \: \vert z \vert < 1-\epsilon\}$
for every $\epsilon>0$. Thus, it preserves arbitrary thin annuli $\{z\in \C_p\; ; \; 1-\epsilon\vert z \vert \leq 1\}$. (Here $\C_p$
is the completion of the algebraic closure of $\Q_p$) 
\end{rem}

\subsubsection{} The derived subgroup of $\Diff(\C,0)$ is the kernel of the morphism 
$\jj_1\colon f\mapsto f'(0)$: 

\begin{thm}
Let $\bfk$ be a complete, non-discrete valued field.
An element $f$ of $\Diff(\bfk,0)$ is a commutator if and only if $f'(0)=1$.
All higher terms of the lower
central series coincide with the kernel of $\jj_1\colon \Diff(\bfk,0)\to \Jet_1(\bfk,0)$. 
\end{thm}

\begin{proof} If $f$ is a commutator $[g,h]$ then $f'(0)=1$. If $f'(0)=1$, compose 
$f$ with the homothety $m_{\lambda}(z)=\lambda z$ for some $\lambda \in \bfk^*$ of
norm $\vert \lambda\vert \neq 1$, and apply Koenigs linearization theorem to find 
an element $h\in \Diff(\bfk,0)$ such that $m_\lambda \circ f = h\circ m_\lambda h^{-1}$ and $h'(0)=1$. 
Then $f=[h, m_\lambda]$. This proves that the derived subgroup of $\Diff(\bfk,0)$ is the kernel of $\jj_1$; 
since $h$ is in the kernel of $\jj_1$,  all subsequent terms of the lower central series
coincide with the derived subgroup. \end{proof}

Now, consider the upper central series. The first terms are $\Diff(\C,0)$ and its derived subgroup $\Diff(\bfk,0)^{(1)}$. 
Then comes 
\begin{equation}
\Diff(\bfk,0)^{(2)}:=[\Diff(\bfk,0)^{(1)},\Diff(\bfk,0)^{(1)}].
\end{equation} The group of jets of order $3$ which are tangent to 
identity, i.e. jets of the form $j(z)=z+a_2 z^2+a_3 z^3$ modulo $z^4$, is an abelian group; at the level of formal 
germs, it is known that the kernel of $\jj_3$ in $\widehat{\Diff}(\bfk,0)^{(1)}$ coincides with the
derived subgroup $\widehat{\Diff}(\bfk,0)^{(2)}$ (see~\cite{Camina:2000}, \S 3, for the description of the upper central series of 
$\widehat{\Diff}(\bfk,0)$). We don't know if a similar statement holds for germs of diffeomorphisms:

\medskip

{\bf{Question.-- }} Does the kernel of $\jj_3$ coincide with the second derived subgroup of $\Diff(\C,0)^{(2)}$ ? More generally, 
what is  the upper central series of $\Diff(\C,0)$ ? \\
 
\section{Appendix: Free groups} \label{par:app-free-groups}

The following theorem, and its proof, are strongly inspired by \cite{MRR}. The proof given in \cite{MRR} 
is somewhat difficult because it makes use of a topology on $\Diff(\C,0)$ which is not compatible with
the group law. We adapt the same proof, without reference to such a topology. 

\begin{thm}\label{thm:MRR}
Let $(\bfk, \vert \cdot\vert)$ be a complete, non-discrete valued field. 
Let $f$ and $g$ be elements of $\Diff(\bfk, 0)$ of infinite order. 
Let $w$ be a non-trivial element of the free group $\F_2$.
Then, there is a polynomial germ of diffeomorphism $h$ such that
$w(hfh^{-1},g)\neq \Id$. 

If $w=a^{n_\ell} b^{n_{\ell-1}}\cdots a^{n_2} b^{n_1}$, one can choose $h$ of the form $z+\epsilon z^2 P(z)$ 
with an arbitrarily small $\epsilon$ and a polynomial function $P\in \bfk[z]$ such that  
$\deg(P)\leq (2\ell)!$ and $\vert P(x)\vert \leq 1$ for all $x\in \disk_1$.
\end{thm}

Before proving this result, let us introduce some vocabulary and notation. Write $w$ 
as a reduced word in the generators $a$ and $b$ of the free group: 
\begin{equation}
w=a^{n_\ell} b^{n_{\ell-1}}\cdots a^{n_2} b^{n_1}
\end{equation}
where the $n_i$ are in $\Z\setminus\{ 0\}$, except maybe if $n_1$ or $n_\ell$ is zero,
but conjugating $w$ by a power of $a$, we only need to consider the case $n_1n_\ell\neq 0$.
Set $N=\max \vert n_i\vert$. 

Let $h$ be an element of $\Diff(\bfk,0)$, and set $f_h=h^{-1}\circ f \circ h$.
Let $r>0$ be smaller than the convergence radius of $h$, $f$,  $g$ and their inverses. 
Choose $R>0$  such that all these germs, and all their compositions
of length $\leq 3N\ell$ map $\disk_R$ inside $\disk_{r}$. If $z$ is a point in $\disk_R$, then its orbit under
the action of $f_h$ and $g$ stays in $\disk_r$ for all compositions of these germs 
given by words of length $\leq N\ell$ in $\F_2$, in this situation, we say that the orbit of $z$ is
{\bf{well defined}} up to length $N\ell$. In particular, if we look at the composition $w(f_h,g)$, and pick a point $z$ in $\disk_R$,
we get a sequence of points 
\begin{equation}
z_0=z, \; z_1=g^{n_1}(z_0), \; z_2=f_h^{n_2}(z_1), \; ..., \ z_\ell=w(f_h,g)(z_0).
\end{equation}
To prove the theorem, we construct a triple $(h,R,z)$ such that the orbit of $z$ is well defined and the
$z_i$ are pairwise distinct; in particular, $z_\ell\neq z_0$ and $w(f_h,g)\neq \Id$.

\begin{proof}
We do a recursion on the length $\ell$, proving the existence of a triple $(h,R,z)$ such that the $z_i$ are pairwise distinct for 
$0\leq i\leq \ell$. Since $f$ and $g$ have infinite order, the union of all fixed points of $f^m$ and $g^m$ in $\disk_r$ for
$-N\leq m\leq N$ is a finite set $F$. For $j=1$, we just pick a point $z_0$ sufficiently near the origin with $z_1:=g^{n_1}(z_0)\neq z_0$; the only constraint is
to take $z_0$ in the complement of $F$. The points $z_0$ and $z_1$ will be kept fixed in the recursion.

Assume that a polynomial germ of diffeomorphism $h_k$ has been constructed, in such a way that 
(a) the points $z_0$, $z_1$, $z_2$, ..., $z_{2k}$, and $z_{2k+1}$ are pairwise distinct (we just intialized the recursion for $k=0$), 
and (b) $h_k(z)=z+\epsilon_k R_k(z)$ for some small $\epsilon_k\in \bfk$ and some element $R_k\in \bfk[z]$ of degree $\leq (2k)!$ which is 
divisible by $z^2$. 
Consider a polynomial germ 
\begin{equation}
P_k(z)=z+\eta_k z^2\prod_{j=0}^{2k}(z-z_j) 
\end{equation}
with a small $\eta\in \bfk$; then 
\begin{itemize}
\item $P_k$ fixes $z_j$ for all $j\leq 2k$, 
\item $P_k(z_{2k+1})=a_k \eta_k +b_k$ for some pair $(a_k,b_k)\in \bfk^2$ with $a_k\neq 0$,
\item as $\eta_k$ goes to $0$, the radius of convergence of $P_k$ and its inverse $P_k^{-1}$ go to infinity.
\end{itemize}
If we compose $h_k$ with $P_k$
then $H=h_k\circ P_k$ is a new polynomial germ such that the orbit of $z_0$ under $f_{H}$ and $g$ gives the same
sequence $z_0$, $z_1$, $\ldots$, up to $z_{2k+1}$. 
The next point is
\begin{equation}
z_{2k+2}= f_{H}^{n_{2k+2}}(z_{2k+1})  
\end{equation}
and we want to exclude the possibility $z_{2k+2}\in \{z_0, \ldots, z_{2k+1}\}$; since
\begin{equation}
z_{2k+2}= (P_k^{-1} \circ f_{h_k}^{n_{2k+2}} \circ P_k)(z_{2k+1}) 
\end{equation}
we want to avoid the inclusion
\begin{equation}
f_{h_k}^{n_{2k+2}}(P_k(z_{2k+1})) \subset P_k \{z_0, \ldots, z_{2k+1}\},
\end{equation}
and for that we just need to choose the parameter $\eta_k$ in the definition of $P_k$ in such a 
way that $f_{h_k}^{n_{2k+2}}(P_k(z_{2k+1}))$ is not in $\{z_0, \ldots, z_{2k}\}$ and $P_k(z_{2k+1})$  is not a fixed point of $f_{h_k}^{n_{2k+2}}$.
These constraints are satisfied for all small non-zero values of $\eta_k$ because $f^{n_{2k+2}}_{h_k}$ is not the identity and the coefficient $a_k$
in $P_k(z_{2k+1})=a_k \eta_k +b_k$ is not zero. 

The next point is $z_{2k+3}=g^{n_{2k+3}}(z_{2k+2})$ and we want it to be disjoint from $\{z_0, \ldots, z_{2k+1},z_{2k+2}\}$.
For this, we do a second perturbation of the conjugacy. Let 
\begin{equation}
Q_k(z)=z+\beta_k z^2\prod_{j=0}^{2k+1}(z-z_j) 
\end{equation}
with a small $\beta_k\in \bfk$; then 
\begin{itemize}
\item $Q_k$ fixes $z_j$ for all $j\leq 2k+1$, 
\item $Q_k(z_{2k+2})=c_k \beta_k +d_k$ for some pair $(c_k,d_k)\in \bfk^2$ with $c_k\neq 0$,
\item as $\beta_k$ goes to $0$, the radius of convergence of $Q_k$ and its inverse $Q_k^{-1}$ go to infinity.
\end{itemize}
Now, we set $h_{k+1}=Q_k\circ H$. This does not change the sequence $z_i$ for $0\leq i\leq 2k+1$, but the
last point $z_{2k+2}$ is replaced by $c_k \beta_k +d_k$. Since $g^{n_{2k+3}}\neq \Id$ and $c_k\neq 0$
any non-zero, small enough value of $\beta_k$ assures that  $z_{2k+3}\notin \{z_0, \ldots, z_{2k+1},z_{2k+2}\}$.

To sum up, if we set $h_{k+1}=Q_k\circ P_k\circ h_k$ then the sequence $z_0$, $\ldots$, $z_{2k+3}$ is now 
made of pairwise distinct points. Moreover, when the parameters $\eta_k$ and $\beta_k$ go to zero, the
germ $Q_k\circ P_k$ and its inverse converge uniformly to the identity on the disk $\disk_{2R}$, so we can 
assume that the orbit of $z_0$ is well defined for all composition of $h_{k+1}$, $f$, $g$, and their inverses
of length $\leq 3N\ell$. The germ $Q_k\circ P_k$ is equal to $z+S_k(z)$ where $S_k$ is 
divisible by $z^2$ and $\deg(S_k)\leq (2k+1)\times (2k+2)$. Thus, 
\begin{equation}
h_{k+1}(z)=z+P_{k+1}(z)
\end{equation}
where $z^2$ divides $P_{k+1}$ and 
\begin{equation}
\deg(P_{k+1})\leq \deg(P_k)\times (2k+1)\times (2k+2)\leq (2k+2) !
\end{equation}
This proves the recursion and finishes the proof of the theorem.
\end{proof} 
 
\begin{thm}
Let $(\bfk, \vert \cdot\vert)$ be a complete, non-discrete valued field. 
If $f$ and $g$ are elements of $\Diff(\bfk;0)$ of infinite order, there exists an 
element $h$ of $\Diff(\bfk;0)$ such that $f_h:=h\circ f \circ h^{-1}$ and $g$ 
generate a free group of rank $2$. One can choose $h$ such that $h'(0)=1$.
\end{thm}

Note that Theorem~\ref{thm_F2} is a direct corollary of that result; one just need
to start with $f=\lambda_1z$ or $\lambda_1z+z^2$ if $\lambda_1$ is a root of unity,
and similarly for $g$.

\begin{proof}
Denote by $a_n$ and $b_n$ the coefficients of $f$ and $g$ respectively. 
Let $L\subset \bfk$ be the field generated by the $a_n$ and $b_n$.
Since $\bfk$ is infinite and complete, its transcendental degree over $L$ is infinite: it contains an infinite sequence $(c_i)$ of 
algebraically independent numbers (over the base field of $\bfk$, see~\cite{}). We can moreover assume that all $c_i$
are in the unit disk. Set 
$h_0(z)=\sum_{n\geq 1} c_nz^n$. 

Consider a non-trivial element $w$ of $\F_2$. The $N$-th coefficient $A_N(w(h_0\circ f \circ h_0^{-1}, g))$ is 
a polynomial function in the coefficients of $f$, $g$, and $h$ (see Section~\ref{par:formal-diff-inversion}). If
it vanishes (resp. if it is equal to $1$), then  $A_N(w(h\circ f \circ h^{-1}, g))=0$ (resp. $1$) for all formal diffeomorphisms $h$, because the $c_i$ are
algebraically independent over $\bfk$. Thus, Theorem~\ref{thm:MRR} implies that $w(h_0\circ f \circ h_0^{-1}, g)\neq \Id$, and this
shows that $f_{h_0}:=h_0\circ f \circ h_0^{-1}$ and $g$ generate a free group of rank $2$.

In this argument, we could start with $h_0=z+\sum_{n\geq 2} c_n z^n$, because we can choose the germ $h$ in Theorem~\ref{thm:MRR}
with the additional constraint $h'(0)=1$. 
\end{proof}

 

\bibliographystyle{plain}
\bibliography{references-sgindiff}

\begin{thebibliography}{10}

\bibitem{Baumslag}
Gilbert Baumslag.
\newblock On generalised free products.
\newblock {\em Math. Z.}, 78:423--438, 1962.

\bibitem{BCN}
M.~Berthier, D.~Cerveau, and A.~Lins~Neto.
\newblock Sur les feuilletages analytiques r\'{e}els et le probl\`eme du
  centre.
\newblock {\em J. Differential Equations}, 131(2):244--266, 1996.

\bibitem{BGSS}
Emmanuel Breuillard, Tsachik Gelander, Juan Souto, and Peter Storm.
\newblock Dense embeddings of surface groups.
\newblock {\em Geom. Topol.}, 10:1373--1389, 2006.

\bibitem{Broer-Tangerman}
H.~W. Broer and F.~M. Tangerman.
\newblock From a differentiable to a real analytic perturbation theory,
  applications to the {K}upka {S}male theorems.
\newblock {\em Ergodic Theory Dynam. Systems}, 6(3):345--362, 1986.

\bibitem{Brudnyi:Survey2010}
Alexander Brudnyi.
\newblock Some algebraic aspects of the center problem for ordinary
  differential equations.
\newblock {\em Qual. Theory Dyn. Syst.}, 9(1-2):9--28, 2010.

\bibitem{Brudnyi:preprint}
Alexander Brudnyi.
\newblock Subgroups of the group of formal power series with the big powers
  condition.
\newblock {\em Preprint}, arXiv:1908.04918v1:1–10, 2019.

\bibitem{Camina:2000}
Rachel Camina.
\newblock The {N}ottingham group.
\newblock In {\em New horizons in pro-{$p$} groups}, volume 184 of {\em Progr.
  Math.}, pages 205--221. Birkh\"{a}user Boston, Boston, MA, 2000.

\bibitem{Cerveau:Survey2013}
Dominique Cerveau.
\newblock Quelques probl\`emes en g\'eom\'etrie feuillet\'ee pour les 60
  ann\'ees de l'{IMPA}.
\newblock {\em Bull. Braz. Math. Soc. (N.S.)}, 44(4):653--679, 2013.

\bibitem{CG}
Christophe Champetier and Vincent Guirardel.
\newblock Limit groups as limits of free groups.
\newblock {\em Israel J. Math.}, 146:1--75, 2005.

\bibitem{Glass:1992}
A.~M.~W. Glass.
\newblock The ubiquity of free groups.
\newblock {\em Math. Intelligencer}, 14(3):54--57, 1992.

\bibitem{Herman-Yoccoz:1983}
M.~Herman and J.-C. Yoccoz.
\newblock Generalizations of some theorems of small divisors to
  non-{A}rchimedean fields.
\newblock In {\em Geometric dynamics ({R}io de {J}aneiro, 1981)}, volume 1007
  of {\em Lecture Notes in Math.}, pages 408--447. Springer, Berlin, 1983.

\bibitem{IP}
Yulij~S. Ilyashenko and A.~S. Pyartli.
\newblock The monodromy group at infinity of a generic polynomial vector field
  on the complex projective plane.
\newblock {\em Russian J. Math. Phys.}, 2(3):275--315, 1994.

\bibitem{Jennings}
S.~A. Jennings.
\newblock Substitution groups of formal power series.
\newblock {\em Canadian J. Math.}, 6:325--340, 1954.

\bibitem{Koenigs}
G.~Koenigs.
\newblock Nouvelles recherches sur les \'{e}quations fonctionnelles.
\newblock {\em Ann. Sci. \'{E}cole Norm. Sup. (3)}, 2:385--404, 1885.

\bibitem{Leslie}
J.~Leslie.
\newblock On the group of real analytic diffeomorphisms of a compact real
  analytic manifold.
\newblock {\em Trans. Amer. Math. Soc.}, 274(2):651--669, 1982.

\bibitem{Lyndon:1959}
R.~C. Lyndon.
\newblock The equation {$a^{2}b^{2}=c^{2}$} in free groups.
\newblock {\em Michigan Math. J}, 6:89--95, 1959.

\bibitem{MRR}
J.-F. Mattei, J.~C. Rebelo, and H.~Reis.
\newblock Generic pseudogroups on {$(\Bbb C,0)$} and the topology of leaves.
\newblock {\em Compos. Math.}, 149(8):1401--1430, 2013.

\bibitem{Milnor:book}
John Milnor.
\newblock {\em Dynamics in one complex variable}, volume 160 of {\em Annals of
  Mathematics Studies}.
\newblock Princeton University Press, Princeton, NJ, third edition, 2006.

\bibitem{Olshanskii}
A.~Yu. Ol{\cprime}shanski{\u{\i}}.
\newblock On residualing homomorphisms and {$G$}-subgroups of hyperbolic
  groups.
\newblock {\em Internat. J. Algebra Comput.}, 3(4):365--409, 1993.

\bibitem{Siegel}
Carl~Ludwig Siegel.
\newblock Iteration of analytic functions.
\newblock {\em Ann. of Math. (2)}, 43:607--612, 1942.

\bibitem{Sternberg:1958}
S.~Sternberg.
\newblock On the structure of local homeomorphisms of euclidean n-space.
\newblock {\em Amer. J. Math.}, 80:623--631, 1958.

\bibitem{Szegedy}
Bal\'azs Szegedy.
\newblock Almost all finitely generated subgroups of the {N}ottingham group are
  free.
\newblock {\em Bull. London Math. Soc.}, 37(1):75--79, 2005.

\bibitem{Takens:1973}
F.~Takens.
\newblock Normal forms for certain singularities of vector fields.
\newblock {\em Ann. Inst. Fourier}, 23(8):163--195, 1973.

\bibitem{Waldschmidt:1990}
M.~Waldschmidt.
\newblock Independance alg\'ebrique des nombres de liouville.
\newblock {\em Lecture Notes in Math.}, 1415:225--235, 1990.

\bibitem{White:1988}
Samuel White.
\newblock The group generated by {$x\mapsto x+1$} and {$x\mapsto x^p$} is free.
\newblock {\em J. Algebra}, 118(2):408--422, 1988.

\end{thebibliography}

\end{document}